\documentclass[12pt]{article}

\usepackage{amsmath}
\usepackage{amsfonts}
\usepackage{amssymb}
\usepackage{amsthm}
\usepackage{enumitem}
\usepackage{mathtools}
\usepackage{setspace}
\usepackage{etoolbox}
\usepackage{color}
\usepackage{graphicx}
\usepackage{verbatim}
\usepackage[normalem]{ulem}
\usepackage[colorinlistoftodos]{todonotes}

\newtheoremstyle{mystyle}
{11pt}				
{11pt}				
{\itshape}					
{}					
{\bfseries}			
{}					
{5.5pt}				
{}					

\newtheoremstyle{mystyle2}
{11pt}				
{11pt}				
{}					
{}					
{\bfseries}			
{}					
{5.5pt}				
{}					

\theoremstyle{mystyle}
\newtheorem{theorem}{Theorem}[section]
\newtheorem{definition}[theorem]{Definition}

\newtheorem{proposition}[theorem]{Proposition}
\newtheorem{corollary}[theorem]{Corollary}
\theoremstyle{mystyle2}
\newtheorem{example}[theorem]{Example}

\renewenvironment{proof}[1][Proof.]{\vspace{-11pt} \begin{trivlist}
		\item[\hskip \labelsep {\bfseries #1}]}{\qed \end{trivlist}}

\setitemize{noitemsep, topsep=5.5pt, parsep=5.5pt, partopsep=0pt}

\allowdisplaybreaks

\newcommand{\gap}{\vspace{11pt}}
\newcommand{\smallgap}{\vspace{5.5pt}}

\newcommand{\diag}{\operatorname{diag}}
\newcommand{\Diag}{\operatorname{Diag}}
\newcommand{\tr}{\operatorname{tr}}

\newcommand{\Aut}{\operatorname{Aut}}
\newcommand{\Lie}{\operatorname{Lie}}

\newcommand{\C}{\mathbb{C}}
\newcommand{\R}{\mathbb{R}}

\newcommand{\Sn}{\mathcal{S}^n}

\newcommand{\one}{{\bf 1}}
\newcommand{\V}{{\cal V}}
\newcommand{\U}{{\cal U}}

\newcommand{\Z}{{\cal Z}}
\newcommand{\W}{{\cal W}}

\newcommand{\G}{{\cal G}}

\newcommand{\lc}{\lambda(c)}
\newcommand{\ly}{\lambda(y)}
\newcommand{\lx}{\lambda(x)}
\newcommand{\lu}{\lambda(u)}

\newcommand{\barx}{\overline{x}}

\newcommand{\ip}[2]{\left< #1, #2 \right>}

\title{\bf  Some commutation principles for  optimization problems over transformation groups and  semi-FTvN systems
}
\author{
	M. Seetharama Gowda\\
	Department of Mathematics and Statistics\\
	University of Maryland, Baltimore County\\
	Baltimore, Maryland 21250, USA\\
	gowda@umbc.edu\\and\\
 David Sossa\\
Instituto de Ciencias de la Ingenier\'{i}a\\
Universidad de O’Higgins\\
Rancagua, Chile\\
david.sossa@uoh.cl}
\date{\today}

\begin{document}

\maketitle

\begin{abstract}
	We introduce the concepts of {\it commutativity} relative to a transformation group and {\it strong commutativity} in the setting of a semi-FTvN system and show their appearance as optimality conditions in certain optimization problems. In the setting of a semi-FTvN system (in particular, in an FTvN system), we show that strong commutativity implies commutativity and observe that in the special case of Euclidean Jordan algebra, commutativity and strong commutativity concepts reduce, respectively, to those of operator and strong operator commutativity. We demonstrate that every complete hyperbolic polynomial induces a semi-FTvN system. By way of an application, we describe several commutation principles. 
\end{abstract}

\gap


\smallgap


\section{Introduction}\label{intro} The main objective of the paper is to describe two commutativity concepts and show their relevance in certain commutation principles. 

For the first commutativity concept, we
consider a real inner product space $\V$ and the corresponding space $\mathrm{ L(\V)}$ of all bounded linear transformations equipped with the operator norm. Within $\mathrm{ L(\V)}$, let  $\mathrm{GL(\V)}$ denote the set of all invertible linear transformations. Given a closed subgroup ${\cal G}$ of $\mathrm{GL(\V)}$, let 
\begin{equation}\label{lie algebra defn}
\mathrm{Lie({\cal G}):=\{D\in \mathrm{ L(\V)}: \exp(tD)\in {\cal G}\,\,\mbox{for all}\,\,t\in \R\}},
\end{equation}
where, for a given $D$,  $\mathrm{\exp(tD)}$ denotes the sum of the power series $\sum_{k=0}^{\infty}\frac{t^kD^k}{k!}$ that converges uniformly on bounded sets of $\R$ to an element in $\mathrm{ L(\V)}$.
(When $\V$ is finite-dimensional, $\mathrm{Lie({\cal G})}$ is the Lie algebra of the matrix Lie group ${\cal G}$.) 
Given elements $a$ and $b$ in $\V$, we say that  
\begin{center}
    {\it  $a$  commutes with $b$  relative to ${\cal G}$ if 
$\langle Da,b\rangle=0$\,\,\mbox{for all}\,\,$\mathrm{D\in \Lie({\cal G})}$.}
\end{center}

   As we see below, in the setting of a Euclidean Jordan algebra $\V$ with its (algebra) automorphism group ${\cal G}$, this definition is equivalent to that of operator commutativity.
   In particular, for real symmetric matrices (or complex Hermitian matrices) $A$ and $B$, this reduces to the usual commutativity $AB=BA$. \\
   
   The motivation for introducing the above concept comes from the so-called commutation principles \cite{seeger-orbital, ramirez et al, gowda-jeong-siopt, gowda-commutation-ORL, jeong-sossa-arxiv2024},  where commutativity appears as an optimality condition in certain optimization problems.  The following result - to be elaborated and proved later - provides an illustration:
\begin{theorem}\label{theorem in introduction}
    Consider $\V$ and ${\cal G}$ as above. Let $\Theta:\V\rightarrow \R$ be a Fr\'{e}chet differentiable function and $F:\V\rightarrow \R$ be a ${\cal G}$-invariant function. Let $E$ be a set in $\V$ that is ${\cal G}$-invariant. If $a$ is a local optimizer of the problem
    $$\underset{x\in E}{min/max}\quad \Theta(x)+F(x),$$
    then $a$ commutes with $\Theta^{\prime}(a)$  relative to ${\cal G}.$
\end{theorem}

Our second commutativity concept is defined in the setting of a semi-FTvN system. A semi-FTvN system is a triple $(\V,\W,\lambda)$, where $\V$ and $\W$ are real inner product spaces and $\lambda:\V\rightarrow \W$ is a mapping that is ``norm-preserving" and "inner product expanding":

\begin{equation} \label{semi-ftvn}
    ||\lambda(x)||=||x||\,\,\mbox{and}\,\, 
\big\langle x,y\big\rangle\leq \langle\lambda(x),\lambda(y)\big\rangle\quad(\forall\, x,y\in \V).
\end{equation}

In a semi-FTvN system $(\V,\W,\lambda)$, we define {\it strong commutativity} of elements $x$ and $y$ in $\V$ by the condition 
\begin{equation}\label{strong commutativity defn}
\langle x,y\rangle=\langle \lambda(x),\lambda(y)\rangle.
\end{equation}
For commutativity in this system, we consider the group $\mathrm{Aut(\V,\W,\lambda)}$ of all invertible bounded linear transformations $A:\V\rightarrow \V$ satisfying the condition $\lambda(Ax)=\lambda(x)$ for all $x\in \V$. As $\lambda$ is norm-preserving, this is a closed subgroup of the orthogonal group of $\V$. Considering a closed subgroup ${\cal G}$ of $\mathrm{Aut(\V,\W,\lambda)}$, we define $\mathrm{Lie({\cal G})}$ by (\ref{lie algebra defn}) and formulate the definition of commutativity relative to ${\cal G}$ as we did earlier. We shall see that 
\begin{itemize}
\item 
{\it In a semi-FTvN system, strong commutativity implies commutativity (relative to ${\cal G}$).}
\end{itemize}

 How do semi-FTvN systems arise? First, a semi-FTvN system is a generalization of the so-called FTvN system. An FTvN system (short for Fan-Theobald-von Neumann system) is a triple    $(\V,\W,\lambda)$, where $\V$ and $\W$ are real inner product spaces and  $\lambda:\V\rightarrow \W$ is a mapping satisfying the condition 
\begin{equation} \label{ftvn defn}
	\max \Big\{\! \ip{c}{x} : \, x \in [u] \Big\} = \ip{\lc}{\lu} \quad (\forall\, c, u \in \V)
\end{equation}
with $[u]:=\{x\in \V: \lambda(x)=\lambda(u)\}$ denoting the so-called $\lambda$-orbit of $u$.  It is shown in \cite{gowda-jeong-ftvn}, Section 3, that in such a system, the mapping $\lambda$ satisfies (\ref{semi-ftvn}); thus, every FTvN system is a semi-FTvN system. Numerous examples and properties of FTvN systems are described in the articles \cite{gowda-ftvn, gowda-jeong-ftvn, jeong-gowda-ftvn, ito-lorenzo}. A simple example of an FTvN system is the triple $(\V,\R,\lambda)$, where $\V$ is a real inner product space and $\lambda(x)=||x||$. Non-trivial examples include  
 {\it Euclidean Jordan algebras},  systems induced by {\it complete and isometric hyperbolic polynomials}, and {\it normal decomposition systems (Eaton triples)}, see \cite{gowda-ftvn}.  In the setting of a Euclidean Jordan algebra (carrying the trace inner product), the strong commutativity relation (\ref{strong commutativity defn}) reduces to the {\it strong operator commutativity} of $x$ and $y$ which means that   $x$ and $y$ have their spectral decompositions with respect to a common Jordan frame where their coefficients (eigenvalues) appear in the decreasing order \cite{baes}.  This commutativity is stronger than the {\it operator commutativity} where (we merely require that) two elements have their spectral decompositions with respect to a common Jordan frame. We show, by rephrasing a recent result in \cite{jeong-sossa-arxiv2024}, that 
 \begin{itemize}
\item {\it In a Euclidean Jordan algebra, commutativity relative to the (algebra) automorphism group is equivalent to operator commutativity.}
\end{itemize}

The concept of strong commutativity (\ref{strong commutativity defn}) was initially studied in the setting of FTvN systems (where it was just called commutativity). 
 For example, it was shown in  \cite{gowda-ftvn}, Section 3.1, that in a FTvN system $(\V,\W,\,\lambda)$, for any $c\in \V$, $\phi:\W\rightarrow \R$, and any spectral set $S$ (which is a set of the form $\lambda^{-1}(Q)$ for some $Q\subseteq \W$) in $\V$,
\begin{equation}\label{eq: linear sup}
\sup_{x\in S} \Big\{\! \ip{c}{x} + (\phi \circ \lambda)(x) \Big\} = \sup_{u \in \lambda(S)} \Big\{\! \ip{\lc}{u} + \phi(u) \Big\}
\end{equation} 
with the attainment of one supremum implying the attainment of the other and additionally implying a  relation of the form 
(\ref{strong commutativity defn}). Furthermore,   it was shown in \cite{jeong-gowda-ftvn} that under certain conditions, 
 (\ref{ftvn defn}) is equivalent to the Fenchel conjugate formula 
$ (\phi \circ \lambda)_{S}^{\ast}(z) = \phi_{\lambda(S)}^{\ast}\big( \lambda(z) \big) $
and to the subdifferential relation 
\begin{equation}
y \in \partial_{S} (\phi \circ \lambda)(\barx) \Longleftrightarrow \ly \in \partial_{\lambda(S)} \phi \big( \lambda(\barx) \big) \text{ and $y$ strongly commutes with $\barx$}. 
\end{equation}

\gap

Semi-FTvN systems also arise from complete hyperbolic polynomials. We recall \cite{garding, bauschke et al} that given  a finite-dimensional real vector space $\V$ and a nonzero element  $e\in \V$, a real homogeneous polynomial $p$ 
of degree $n$ on $\V$ is said to be {\it hyperbolic relative to $e$} if $p(e)\neq 0$ and
for every $x\in \V$,   the roots of the univariate polynomial $t\rightarrow p(te-x)$ are  
real. Given such a polynomial $p$,  for any $x\in \V$, let $\lambda(x)$  denote the vector
of roots of the polynomial $t\rightarrow p(te-x)$  with entries written in decreasing
order. We say that $p$ is complete if $\lambda(x)=0\Rightarrow x=0$. Based on a construction due to Bauschke et al \cite{bauschke et al}, we show that 

\begin{itemize}
\item {\it Every complete hyperbolic polynomial induces a semi-FTvN system.}
    \end{itemize}

In addition to Theorem \ref{theorem in introduction} and the bulleted items above, we state and prove several commutation principles.

The organization of the paper is as follows. In Section 2, we introduce the concept of commutativity relative to a transformation group and provide some examples. Section 3 and its various subsections deal with semi-FTvN systems, automorphisms, and commutativity concepts, and the result that strong commutativity implies commutativity. In this section, we also introduce the concepts of center and unit element, and provide numerous examples.  In Section 4, we consider the FTvN system induced by a Euclidean Jordan algebra. Here, we show that commutativity relative to the 
(algebra) automorphism group is equivalent to operator commutativity. In Section 5, we show that a complete hyperbolic polynomial induces a semi-FTvN system. 
In our final section, we describe several commutativity principles. These can be regarded as analogs of earlier results stated in the setting of FTvN systems that describe strong commutativity relations.  
\section{Commutativity relative to a transformation group}
Consider a real inner product space $\V$ and the space $\mathrm{ L(\V)}$ of all bounded linear transformations on $\V$  equipped with the operator norm. Within $\mathrm{ L(\V)}$, let  $\mathrm{GL(\V)}$ denote the set of  all invertible linear transformations. A closed subgroup of $\mathrm{GL(\V)}$ is said to be a {\it transformation group} (a matrix Lie group \cite{hall} when $\V$ is finite dimensional). Given such a group ${\cal G}$, we consider $\Lie({\cal G})$ defined in (\ref{lie algebra defn}).
\begin{definition}
   Let $\V$ and ${\cal G}$ be as above and $a,b\in \V$.  We say that  $a$ commutes with $b$ relative to ${\cal G}$ if 
    $\langle Da,b\rangle=0\,\,\mbox{for all}\,\,D\in \Lie({\cal G}).$
    \end{definition}
We remark that the above definition is not symmetric in $a$ and $b$. However, when ${\cal G}$ is a closed subgroup of the orthogonal group of $\V$ (consisting of elements of $\mathrm{GL(\V)}$ that are inner product preserving), from the relation
$\langle e^{tD}x,e^{tD}x\rangle=\langle x,x\rangle$, we get $\langle Dx,x\rangle=0$ and \begin{equation}\label{skew-symmetry}
\langle Dx,y\rangle+\langle x,Dy\rangle=0\,\,\mbox{for all}\,\,x,y\in \V.
\end{equation}
Consequently,
in this setting, {\it if $a$  commutes with $b$ relative to ${\cal G}$, then $b$ commutes with $a$ relative to ${\cal G}$}. When this happens, we say that {\it $a$ and $b$ commute relative to ${\cal G}$.}

\gap

 While the above definition is formulated in the setting of any real inner product space, we will be mostly concerned with finite-dimensional real inner product spaces and matrix Lie groups. 
For matrix Lie groups, we have the following result. Here, $I$ denotes the Identity transformation/matrix, $(-\delta,\delta)$ is an open interval in $\R$ and for a differentiable function $\gamma$, $\gamma^\prime$ denotes the derivative. This result shows that one can define commutativity via the tangent space/cone at the identity in ${\cal G}$.

\begin{theorem} \label{tangent cone version of lie algebra} Let ${\cal G}$ be a matrix Lie group over a  finite-dimensional real inner product space $\V$. For any $D\in \mathrm{L(\V)}$, the following are equivalent:
\begin{itemize}
    \item [(a)] $D\in \Lie({\cal G})$, that is, 
     $\exp(tD)\in {\cal G}$ for all $t\in \R$.
    \item [(b)] $D$ belongs to the tangent space of ${\cal G}$ at $I$, that is, there exists a smooth/differentiable curve $\gamma: (-\delta,\delta)\rightarrow {\cal G}$ such that 
    $\gamma(0)=I\quad\mbox{and}\quad \gamma^{\prime}(0)=D.$
    \item [(c)] $D$ belongs to the Bouligand tangent cone of ${\cal G}$ at $I$, that is, there exist sequences $\{D_k\}$ in ${\cal G}$ and $\{t_k\}$ in $\R$ such that $D_k\rightarrow I$, $t_k\downarrow 0$ and $\frac{D_k-I}{t_k}\rightarrow D$.
\end{itemize}
Thus,
\begin{center}
{\it     $\Lie({\cal G})$ = Tangent space of ${\cal G}$ at $I$= (Bouligand) tangent cone of ${\cal G}$ at $I$.
}
\end{center}
\end{theorem}
In the above, the equivalence of $(a)$ and $(b)$ is well-known, see e.g., \cite{hall}, Corollary 2.35 and \cite{stillwell}, pages 93 and 104. \\The equivalence of $(b)$ and $(c)$ is given in 
\cite{jeong-sossa-arxiv2024}. \\

Suppose $\V$ is a finite-dimensional real inner product space. For any  $A\in \mathrm{L(\V)}$, let $A^T$ denote the transpose of $A$ defined by $\langle A^Tx,y\rangle =\langle x,Ay\rangle$; let the trace of $A$ be given by 
 $\tr(A):=\sum \langle Ae_i,e_i\rangle$, where $\{e_1,e_2,\ldots, e_n\}$ is  any fixed orthonormal basis of $\V$.  We make $\mathrm{L(\V)}$ into a real inner product space via 
$\langle A,B \rangle: =\tr(AB^T)$. Also, for $a,b\in \V$, we define $a\otimes b$ by $(a\otimes b)(x):=\langle b,x\rangle a$. Then, we can describe commutativity in $\V$ as follows: 
\begin{center}
  {\it   $a$ commutes with $b$ relative to ${\cal G}$ if and only if $b\otimes a\perp \Lie({\cal G}).$}
\end{center}

In view of the above theorem,  the condition $b\otimes a\perp \Lie({\cal G})$ reduces to the `orbital relation' of Seeger \cite{seeger-orbital} and, in the case of a Euclidean Jordan algebra, reduces to operator commutativity, see \cite{jeong-sossa-arxiv2024}.

\smallgap

We now list/describe some examples. More will be given in later sections.

\smallgap

\begin{example}
Let $\V=\R^n$ (carrying the usual inner product) with ${\cal G}$ denoting the set of all $n\times n$ permutation matrices. As ${\cal G}$ is finite, $\Lie({\cal G})=\{0\}$; so, relative to ${\cal G}$,  any two elements in $\R^n$  commute.
\end{example}

\begin{example}
Let $\V=\R^n$ (carrying the usual inner product) and ${\cal G}=\mathrm{GL}(\R^n)$. Then 
$\Lie({\cal G})=\R^{n\times n}$ and so if $a$  commutes with $b$ relative to ${\cal G}$, then $a=0$ or $b=0$.
\end{example}

\begin{example}\label{example: cone automorphism group}
In a real finite-dimensional inner product space $\V$, consider a proper cone $K$ (so $K$ is a closed convex cone which is pointed and has nonempty interior) with dual  $K^*:=\{y\in \V: \langle y,x\rangle\geq 0\,\,\mbox{for all}\,\,x\in K\}.$
Let ${\cal G}$ be the automorphism group of $K$, i.e.,
$${\cal G}=\mathrm{Aut(K)}:=\{L\in \mathrm{ L(\V)}: L(K)=K\}.$$
For any $L\in \mathrm{ L(\V)}$, Schneider and Vidyasagar \cite{schneider-vidyasagar} have shown that 
$$\exp(tL)(K)\subseteq K\,\,\mbox{for all}\,\,t\geq 0\,\,\mbox{if and only if}\,\,[x\in K,y\in K^*,\langle x,y\rangle=0]\Rightarrow \langle L(x),y\rangle \geq 0.$$
 From this, it follows that the Lie algebra of $\mathrm{Aut(K)}$ is the space of all linear transformations $L$ satisfying the condition
\begin{equation}\label{lyapunov-like}
[x\in K,y\in K^*,\langle x,y\rangle=0]\Rightarrow \langle L(x),y\rangle =0.
\end{equation}
In the optimization/complementarity literature,  a linear transformation $L$ satisfying the condition (\ref{lyapunov-like}) is said to be {\it Lyapunov-like}.
So, in the present setting, given $a,b\in \V$,
\begin{center}
     {\it $a$ commutes with $b$ relative to $\mathrm{Aut(K)}$ if and only if \\$\langle L(a),b\rangle=0$ for all Lyapunov-like transformations $L$ on $K$.}
\end{center}
For example, if $(a,b)$ is a {\it complementary pair}, that is, if $a\in K$, $b\in K^*$ and $\langle a,b\rangle=0$, then $a$  commutes with $b$ relative to $\mathrm{Aut(K)}$. 
As $L\in \mathrm{Aut(K)}\Leftrightarrow L^T\in \mathrm{Aut(K^*)}$, where $L^T$ denotes the transpose of $L$, we see that 
$a$  commutes with $b$ relative to $\mathrm{Aut(K)}$ if and only if $b$ commutes with $a$ relative to $\mathrm{Aut(K^*)}$. 
\\
As a further illustration, consider $\V=\R^n$ with the usual inner product and let $K=\R^n_+$. It is known (and can easily be verified) that Lyapunov-like transformations on $\R^n_+$ are just diagonal matrices. So, in this setting, $a$ commutes with $b$ if and only if 
$\langle Da,b\rangle=0$ for every diagonal matrix $D$. This simplifies to:
{\it $a$ commutes with $b$ relative to $\mathrm{Aut(\R^n_+)}$ if and only if $a*b=0$}, where $a*b$ is the componentwise product of $a$ and $b$.
\end{example}

\smallgap

\begin{example}
Consider a Euclidean Jordan algebra $\V$ of rank $n$ and unit element $e$. We assume that $\V$ carries the trace inner product. (For various definitions, see Section 3.2 below.) Let $K$ be the corresponding symmetric cone and ${\cal G}:=\mathrm{Aut(K)}$ (as in the previous example). In this setting, it is known (\cite{tao-gowda-representation}, Theorem 1) that Lyapunov-like transformations are of the form 
$$L=L_c+D,$$
where $c\in \V$, $L_c(x)=x\circ c$, and $D$ is a derivation, so
$D(x\circ y)=(Dx)\circ y+x\circ Dy$ for all $x,y\in \V$.
Then $a$ commutes with $b$ relative to $\mathrm{Aut(K)}$ if and only if $\langle c\circ a,b\rangle=0$ and $\langle Da,b\rangle=0$ for all $c\in \V$ and derivations $D$. Now $0=\langle c\circ a,b\rangle=\langle c,a\circ b\rangle$ for all $c$ implies that $a\circ b=0$. As we shall see in the proof of Theorem \ref{ commutativity equals operator commutativity}, the condition that $\langle Da,b\rangle=0$ for all derivations $D$ is equivalent to the operator commutativity of $a$ and $b$. In conclusion,  
\begin{center}
    {\it $a$ commutes with $b$ relative to $\mathrm{Aut(K)}$ if and only if \\
$a$ and $b$ operator commute and $a\circ b=0.$}
\end{center}
\end{example}

\begin{example}
    Consider the space $\Sn$ of all $n\times n$ real symmetric matrices with inner product 
    $\langle X,Y\rangle: =\tr(XY)$ (the trace of $XY$) and Jordan product 
    $X\circ Y:=\frac{XY+YX}{2}$.  Let $K$ be the completely positive cone in $\Sn$, i.e., 
    $$K=\Big \{\sum uu^T: u\in \mathbb R^n,\,u\geq 0\Big \},$$
    where $\sum uu^T$ denotes a finite sum of matrices of the form $uu^T$. Let ${\cal G}$ be the automorphism group of $K$ (as defined in Example \ref{example: cone automorphism group}). It is known, see \cite{gowda et al - lie algebra}, Example 2, that every element of $\mathrm{Lie(Aut(K))}$ is a transformation of the form $L_A$, where $L_A(X)=AX+XA^T$ and $A$ is an $n\times n$ diagonal matrix. Hence, in this setting, $X$ commutes with  $Y$  relative to $\mathrm{Aut(K)}$ if and only if $\langle L_A(X),Y\rangle=0$ for all $n\times n$ diagonal matrices $A$. 
    Now, the latter condition  simplifies to: $0 =\tr(AXY+XAY)=2\,\tr(A(X\circ Y))$ for all diagonal matrices $A$ and, subsequently, to $ \diag(X\circ Y)=0$. We conclude that  $X$ commutes with $Y$ relative to $\Aut(K)$ if and only if $ \diag(X\circ Y)=0$. 
\end{example}

\smallgap

Going beyond $\Sn$, we now consider the space $M_n$ of all $n\times n$ real matrices with inner product $\langle X,Y\rangle  =\tr(XY^T)$ and corresponding norm $||\cdot||.$  In $M_n$, let ${\cal O}(n)$ denote the orthogonal group of $\R^n$, so,
$${\cal O}(n):=\{A\in M_n: AA^T=I\}.$$
Let ${\cal G}_1$ and ${\cal G}_2$ denote closed subgroups of ${\cal O}(n)$. Within $\mathrm{L}(M_n)$, consider the matrix Lie group $\cal G$ of all linear transformations of the form 
$$T:X\mapsto UXV\quad (X\in M_n),$$
where $U\in {\cal G}_1$ and $V\in {\cal G}_2$. In the next result, we describe the Lie algebra of ${\cal G}$ and commutativity relative to ${\cal G}$.

\begin{theorem}\label{th:square}
Let ${\cal G}_1$, ${\cal G}_2$, and ${\cal G}$ be as above. 
Then,
\begin{itemize}
    \item  [$(i)$]  $\Lie({\cal G})$ consists of transformations of the form
$X\mapsto AX+XB$,
where $A\in \Lie({\cal G}_1)$ and $B\in \Lie({\cal G}_2).$ 
 \item [$(ii)$] Given $X,Y\in M_n$, $X$ commutes with $Y$ relative to ${\cal G}$ if and only if $YX^T\perp 
 \Lie({\cal G}_1)$ and $X^TY\perp \Lie({\cal G}_2)$.
 \end{itemize}
\end{theorem}

\begin{proof}
 In what follows, we reserve the letters $U,U_1,U_2, U(t)$, etc., for matrices in ${\cal G}_1$ and $V,V_1,V_2,V(t),$ etc., for  matrices in ${\cal G}_2$. Since a scalar multiple of the identity matrix $I$ in $M_n$ is the only matrix that commutes with all other matrices in $M_n$, and that ${\cal G}_1$ and ${\cal G}_2$ are subsets of ${\cal O}(n)$, we see that the equality $UXV=X$ holds for all $X\in M_n$ if and only if $U=V=I$ or $U=V=-I$. Consequently,  the equality $U_1XV_1=U_2XV_2$ holds for all $X\in M_{n}$ if and only if either $U_1=U_2$ and $V_1=V_2$, or $U_1=-U_2$ and $V_1=-V_2$.     \\
 
$(i)$ Let $D\in \mathrm{Lie({\cal G})}$ so that for every  $t\in \R$, there exist matrices $U(t)\in {\cal G}_1$ and $V(t)\in {\cal G}_2$ such that 
$$\exp(tD)X=U(t)XV(t)\quad (X\in M_n).$$
 In particular, for $t=0$, we have $X=U(0)XV(0)$ for all $X$. By our observation above, either $U(0)=V(0)=I$ or $U(0)=V(0)=-I$. Let $\varepsilon>0$ so that relative to the operator norm on $M_n$, the open balls $B(I,\varepsilon)$ and  $B(-I,\varepsilon)$) around $I$ and $-I$  are disjoint. We now claim the existence of $\delta>0$ such that 
     either for all $t\in (-\delta,\delta)$,  $U(t)\in B(I,\varepsilon)$ or for all $t\in (-\delta,\delta)$,  $U(t)\in B(-I,\varepsilon)$. 
     This can be seen as follows: Suppose there is an $\varepsilon>0$ and a sequence $(t_k)$ of positive numbers converging to zero such that for all $k$, $||U(t_k)-I||\geq \varepsilon$ and $||U(t_k)+I||\geq \varepsilon$. Then, there exists a subsequence of $(t_k)$ so that $U(t_{k_j})\rightarrow U\in {\cal G}_1$ and $V(t_{k_j})\rightarrow V\in {\cal G}_2$
as $j\rightarrow \infty$. It follows that $||U-I||\geq \varepsilon$ and $||U+I||\geq \varepsilon$. However, letting $t_{k_j}\rightarrow 0$ in  $\exp(t_{k_j}D)X=U(t_{k_j})XV(t_{k_j})$, we get  $I=UXV$ for all $X$, and, consequently, $U=V=I \in B(I,\varepsilon)$ or $U=V=-I\in B(-I,\varepsilon)$. We reach a contradiction.\\
By a similar argument, we can assume the existence of a $\delta>0$ (without loss of generality, the same $\delta$ as before) such that either for all $t\in (-\delta,\delta)$,  $V(t)\in B(I,\varepsilon)$ or for all $t\in (-\delta,\delta)$,  $V(t)\in B(-I,\varepsilon)$.  Using the facts that $U(0)=V(0)$ and $U(t)XV(t)=(-U(t))X(-V(t))$, we may assume, without loss of generality, 
\begin{center}
     $U(t),V(t)\in B(I,\varepsilon)$ for all $t\in (-\delta,\delta)$.
    \end{center}
 Because of our choice of $\epsilon$, for each $t\in (-\delta, \delta)$, $U(t)$ and $V(t)$ are uniquely defined in    $B(I,\varepsilon)$. Additionally, 
 a modification of the above argument shows the continuity of  $U(t)$ and $V(t)$ at $0$.
\\
We now show that $U(t)$ and $V(t)$ (chosen above) are differentiable by showing that their columns are differentiable. We present an argument for the first column of $U(t)$ (with a similar argument for other columns).\\
Now, for any fixed $X$,  $\exp(tD)X$ is differentiable in $t$; hence its transpose $\Big (\exp(tD)X)\Big)^T$
is differentiable. It follows that 
$$U(t)XX^TU(t)^T=\Big (\exp(tD)X\Big)\Big (\exp(tD)X)\Big)^T $$ is also differentiable. 
 Now, let  $X=e_1e_1^T$, where $e_1$ is the column vector in $\R^n$ with $1$ in the first slot and zeros elsewhere. Let $u_1(t):=U(t)e_1$ denote the first column of $U(t)$. Then, each entry in the (matrix) function 
$$u_1(t)u_1(t)^T=U(t)XX^TU(t)^T,$$
namely, $a_{ij}(t):=(u_1(t))_i(u_1(t))_j$
is differentiable. As $U(0)=I$, we see that 
$a_{11}(t)=[(u_1(t))_1]^2$ is differentiable in $t$ (in particular, near zero) with $a_{11}(0)=1$.  It follows from the continuity of $U(t)$ at $0$ that for all $t$ near zero, $\alpha(t):=(u_1(t))_1=\sqrt{a_{11}(t)}$ is well-defined, positive, and differentiable. Hence, for any $j$, $(u_1(t))_j=\frac{1}{\alpha(t)}a_{1j}(t)$  is also differentiable near zero.\\
In summary, we have shown that $U(t)$ is differentiable near $t=0$. A similar argument works for $V(t)$ as well. 
\\
Now, for any $X\in M_n$, differentiating   $\exp(tD)X=U(t)XV(t)$ at $t=0$, we obtain
$$D(X)=U^\prime(0)X+XV^\prime(0).$$ Since $U(0)=I=V(0)$, we see that $A:=U^\prime(0)\in \Lie({\cal G}_1)$ and $B:=V^\prime(0)\in \Lie({\cal G}_2)$. Hence, $D(X)=AX+XB$ for all $X$ proving the first part of $(i)$. \\
To see the converse, consider the transformation 
$T:M_n\rightarrow M_n$ defined by $T(X)=AX+XB$, where $A\in \Lie({\cal G}_1)$ and $B\in \Lie({\cal G}_2)$. In view of Theorem \ref{tangent cone version of lie algebra}, we find two differentiable curves $p(t)$ and $q(t)$ over some interval $(-\delta, \delta)$ around zero into ${\cal O}(n)$ such that $p(0)=I=q(0)$ with $p^\prime(0)=A$ and $q^\prime(0)=B$. Then $h(t): (-\delta,\delta)\rightarrow L(\V)$, defined by  $h(t)X:=p(t)Xq(t)$ for all $X\in M_n$, is a differentiable curve with $h(0)=I$ and $h^\prime(0)=T.$ This proves that 
$T\in \Lie({\cal G})$. \\
$(ii)$ Suppose $X$ and $Y$ commute relative to ${\cal G}$. Then, by our definition, $\langle T(X),Y\rangle=0$ for all $T\in \Lie({\cal G})$. Let $A\in \Lie({\cal G}_1)$ and $B\in \Lie({\cal G}_2)$ be arbitrary  so that the transformation $T$ defined by $X\mapsto AX+XB$ is in $\Lie({\cal G})$. Then, $\langle AX+XB,Y\rangle=0$ for all $A$ and $B$  specified above. By taking $B=0$, we see that $\langle A,YX^T\rangle =\tr(AXY^T)=\langle AX,Y\rangle=0$ for all  $A\in \Lie({\cal G}_1)$. This shows that $YX^T\perp \Lie({\cal G}_1).$  Likewise, by putting $A=0$ and working with $B$, we see that $X^TY\perp \Lie ({\cal G}_2)$.\\
Conversely, if $YX^T\perp \Lie({\cal G}_1)$ and  $X^TY\perp \Lie ({\cal G}_2)$, then $\langle AX+XB,Y\rangle=0$ for all $A\in \Lie({\cal G}_1)$ and $B\in \Lie({\cal G}_2)$. This shows, by $(i)$, that $\langle T(X),Y\rangle=0$ for all $T\in \Lie({\cal G})$. This completes the proof.
\end{proof}

Now suppose ${\cal G}_1={\cal O}(n)={\cal G}_2$. As the Lie algebra of ${\cal O}(n)$ is the space of all $n\times n$ (real) skew-symmetric matrices and the orthogonal complement of this space in $M_n$ is the space of all $n\times n$ real symmetric matrices, we have the following result.

\begin{corollary}\label{commutativity in Mn}
 \label{commutativity in Mn}
On $M_n$, consider the group ${\cal G}$ of linear transformations of the form 
$X\mapsto UXV$,
where $U$ and $V$ are orthogonal. Then,
\begin{itemize}
    \item  [$(i)$]  $\Lie({\cal G})$ consists of transformations of the form
$X\mapsto AX+XB$,
where $A$ and $B$ are skew-symmetric. 
\item [$(ii)$] $X,Y\in M_{n}$ commute relative to ${\cal G}$ if and only if $XY^T$ and $X^TY$ are symmetric.
\end{itemize}

    \end{corollary}

\noindent{\bf Remarks.} For the sake of simplicity of presentation, Theorem\,\ref{th:square} deals with transformations on the space of square matrices $M_n$. But, with minor modifications, this theorem can be also adapted for transformations on the space of $m\times n$ rectangular matrices $M_{m,n}$. For example, analogous to Corollary \ref{commutativity in Mn}, for the group $\mathcal G$ of linear transformations on $M_{m,n}$ of the form $X\mapsto UXV$ with $(U,V)\in\mathcal O(m)\times\mathcal O(n)$, $\Lie({\cal G})$ consists of transformations of the form
$X\mapsto AX+XB$,
where $A$ and $B$ are skew-symmetric on $M_m$ and $M_n$, respectively. A similar result was presented by Lewis and Sendov \cite{lewis-sendov}.

\section{Semi-FTvN systems}
In this section, we formally introduce semi-FTvN systems and describe some examples and results. 
In what follows, for simplicity, we use the same inner product and norm notation. 

\begin{definition}\label{defn of semi-ftvn system}
    We say that a triple $(\V,\W,\lambda)$ is a semi-FTvN system if $\V$ and $\W$ are real inner product spaces with $\lambda:\V\rightarrow \W$ satisfying the following conditions:
    \begin{itemize}
        \item [(A1)] $\lambda$  is norm-preserving, that is, $||\lambda(x)||=||x||$ for all $x\in \V$;
        \item [(A2)] $\lambda$ is inner product expanding, that is, $\langle x,y\rangle\leq \big\langle \lambda(x),\lambda(y)\big\rangle$ for all $x,y\in \V$.
    \end{itemize}
\end{definition}

 We recall that (\ref{ftvn defn}) defines a FTvN system.
In view of (\ref{ftvn defn}) implying (\ref{semi-ftvn}), every FTvN system is a semi-FTvN system. The following statements are easy to verify:
\begin{itemize}
\item $(\V,\W,\lambda)$ is a semi-FTvN system if and only if the following `sharpened' Cauchy-Schwarz inequality holds for all $x,y\in \V$: $\langle x,y\rangle\leq \langle \lambda(x),\lambda(y)\rangle \leq ||x||\,||y||.$
\item  
 If $(\V,\W,\lambda)$ is a semi-FTvN system, then so is $(\V,\W,\widetilde{\lambda}).$
where $\widetilde{\lambda}(x):=-\lambda(-x)$.
\item 
$(\V,\W,\lambda)$ is a semi-FTvN system and $\U$ is a (linear) subspace of $\V$, then $(\U,\W,\lambda)$, with $\lambda$ restricted to $\U$, is also a semi-FTvN system. (A similar statement for FTvN systems is false, see Example \ref{subspace of a FTvN} below.)
\item  If $(\V,\W,\lambda)$ and $(\W,\Z,\mu)$ are  semi-FTvN systems, then so is their `composition' $(\V,\Z,\mu\circ \lambda)$. (A similar statement for FTvN systems is false, see Example \ref{composition example} below.)
\item  If $(\V_1,\W_1,\lambda_1)$ and $(\V_2,\W_2,\lambda_2)$ are  semi-FTvN systems, then so is their `product' $(\V_1\times\V_2, ,\W_1\times \W_2,\lambda_1\times \lambda_2)$. 
\item  If $(\V,\W,\lambda)$  is a semi-FTvN system, then so is $(\overline{\V},\overline{\W},\overline{\lambda})$, where $\overline{\V}$ and $\overline{\W}$ are the Hilbert space completions of $\V$ and $\W$ respectively, and 
$\overline{\lambda}$ is the unique extension of $\lambda.$ (We note that $\lambda$ is Lipschitzian on $\V$, hence uniformly continous.) It is not known if such a statement holds for FTvN systems, see Problem 2 on page 12 of \cite{gowda-jeong-ftvn}.
\end{itemize}

The following properties are known in the setting of FTvN systems \cite{gowda-ftvn}. They continue to hold in semi-FTvN systems.
\begin{proposition} \label{elementary properties of lambda} Suppose $(\V,\W,\lambda)$ is a semi-FTvN system.  Then the following hold for all $x,y\in \V$ and $\alpha\in \R$:
\begin{itemize}
    \item [$(a)$] $\lambda(\alpha\,x)=\alpha\,\lambda(x)$ for all $\alpha\geq 0$.
\item [$(b)$] $||\lambda(x)-\lambda(y)||\leq ||x-y||.$
\item [(c)] $||\lambda(x+y)||\leq ||\lambda(x)+\lambda(y)||.$
\end{itemize}
\end{proposition}

\noindent{\it Note:} Item $(b)$ shows that $\lambda$ is Lipschitz continuous.

\smallgap

\begin{proof}
$(a)$  Let $\alpha\geq 0$. Using $(A1)$ and $(A2)$ (of Definition \ref{defn of semi-ftvn system}), we have  
$$||\lambda (\alpha\,x)-\alpha\lambda(x)||^2=||\lambda (\alpha\,x)||^2+\alpha^2||\lambda(x)||^2-2\alpha\langle \lambda(\alpha x),\lambda(x)\rangle \leq ||\alpha\,x||^2+\alpha^2||x||^2-2\alpha\langle \alpha x,x\rangle=0.$$ This proves $(a)$.\\ 
$(b)$ Using $(A1)$ and $(A2)$, we have   
$||\lambda(x)-\lambda(y)||^2-||x-y||^2=
2[\langle x,y\rangle-\langle \lambda(x),\lambda(y)\rangle]\leq 0$. The stated Lipschitzian property follows.\\
$(c)$ From $(A1)$, $||\lambda(x+y)||^2=||x+y||^2=\langle x+y,x+y\rangle=\langle x+y,x\rangle+\langle x+y,y\rangle$.\\ From $(A2)$, 
$\langle x+y,x\rangle \leq \langle \langle \lambda(x+y),\lambda(x)\rangle$ and 
$\langle x+y,x\rangle \leq \langle \langle \lambda(x+y),\lambda(y)\rangle$; hence, 
\begin{equation}\label{CS}||\lambda(x+y)||^2\leq \langle \lambda(x+y),\lambda(x)+\lambda(y)\rangle\leq ||\lambda(x+y)||\,||\lambda(x)+\lambda(y)||,
\end{equation}
where the second inequality is due to the Cauchy-Schwarz inquality.
This gives $(c)$. 
\end{proof}

\subsection{Automorphisms and  commutativity}

\begin{definition}\label{automorphism defn}
    In the setting of a semi-FTvN system $(\V,\W,\lambda)$, 
     an invertible linear transformation $A:\V\rightarrow \V$ is said to be an automorphism if $\lambda(Ax)=\lambda(x)$ for all $x\in \V$. 
\end{definition}
Given a semi-FTvN system $(\V,\W,\lambda)$, we let $\mathrm{Aut(\V,\W,\lambda)}$ (or $\mathrm{Aut(\V)}$, for short) denote the set of all automorphisms. Since $\lambda$  is norm-preserving, each automorphism is a bounded linear transformation and  an orthogonal transformation (i.e., inner product preserving); so $\mathrm{Aut(\V)}\subseteq {\cal O}(\V),$
where ${\cal O}(\V)$ ($=$ the set of all invertible orthogonal linear transformations) denotes the orthogonal group of $\V$. As in (\ref{lie algebra defn}), we let
$$\mathrm{Lie(Aut(\V))}:=\{D\in \mathrm{ L(\V)}: \exp(tD)\in \mathrm{Aut(\V)}\,\,\mbox{for all}\,\,t\in \R\}.$$

\gap

  While many of our concepts and results below are based on the group $\mathrm{Aut}(\V)$, in some settings, it may be more advantageous to work with a closed subgroup ${\cal G}$ of  $\mathrm{Aut}(\V)$ (such as the connected component of identity in $\mathrm{Aut}(\V)$ when $\V$ is finite-dimensional or the one that appears in Example \ref{NDS} below).   In this case, 
$$\mathrm{Lie({\cal G})}:=\{D\in \mathrm{ L(\V)}: \exp(tD)\in {\cal G}\,\,\mbox{for all}\,\,t\in \R \}\,\subseteq\, \mathrm{Lie(Aut(\V))}. $$

\gap

We now define commutativity in a semi-FTvN system. Our motivation comes from Remark 3.7 in \cite{gowda-jeong-siopt} and a recent formulation of operator commutativity in Euclidean Jordan algebras (\cite{jeong-sossa-arxiv2024}, Proposition 2).

\begin{definition}\label{weak commutativity defn}
    Let  $(\V,\W,\lambda)$ be a semi-FTvN system and ${\cal G}$ be a closed subgroup of $\mathrm{Aut(\V)}$. Given $a,b\in \V$, 
    we say that
    \begin{itemize}
        \item [$(i)$] $a$  strongly commutes with $b$ if $\big\langle a,b\big\rangle=\big\langle \lambda(a),\lambda(b)\big\rangle.$
    \item [$(ii)$] $a$ commutes with $b$ if $\langle Da,b\rangle=0$ for all $D\in \mathrm{Lie(Aut(\V))}.$
    \item [$(iii)$]  $a$ commutes with $b$ relative to ${\cal G}$ if 
$\langle Da,b\rangle=0$ for all $D\in \mathrm{Lie({\cal G})}$.
    \end{itemize}
\end{definition}

\noindent{\bf Remarks.} $(\alpha)$ In our previous work  on FTvN systems \cite{gowda-ftvn, gowda-jeong-ftvn, jeong-gowda-ftvn}, we used the condition $\big\langle a,b\big\rangle=\big\langle \lambda(a),\lambda(b)\big\rangle$ to define commutativity. In order to use the terminologies compatible with those in Euclidean Jordan algebras, henceforth we use the new term `strong commutativity'. \\
$(\beta)$ {\it The above definitions are symmetric in $a$ and $b$:} The first definition is clearly symmetric in $a$ and $b$. From (\ref{skew-symmetry}), $\langle Da,b\rangle=0\Rightarrow \langle Db,a\rangle=0$ when $D\in \mathrm{Lie({\cal G}});$ so, the other two definitions are also symmetric.  Because of this symmetry, we use phrases such as `$a$ and $b$ (strongly) commute' etc., while dealing with semi-FTvN systems. \\
We note that if $a$ and $b$  commute (relative to ${\cal G}$), then $\pm a$ and $\pm b$ commute (relative to ${\cal G})$. Also, {\it commutativity implies commutativity relative to any ${\cal G}$} as $\mathrm{Lie({\cal G})}\subseteq \mathrm{Lie(Aut(\V))}.$\\
$(\gamma)$ While strong commutativity of $a$ and $b$ is easily verifiable, checking commutativity requires the knowledge of the automorphism group $\mathrm{Aut(\V)}$ (rather, its connected component of the Identity transformation) and its Lie algebra.  \\
$(\delta)$ The relation 
$\langle Da,b\rangle=0$ for all $D\in \mathrm{Lie(Aut({\cal G}))}$ is a type of `orbital relation' defined in \cite{seeger-orbital}. \\
$(\epsilon)$ Suppose $\V$ is finite dimensional so that ${\cal G}$ is a matrix Lie group. It is well known that the connected component ${\cal G}_0$ of ${\cal G}$ is also a matrix Lie group (cf. \cite{hall}, Prop. 1.8). Hence, if $a$ and $b$ commute relative to ${\cal G}$, then they commute relative to ${\cal G}_0$. 

\gap

In the theorem below, we show that strong commutativity implies commutativity. First, we describe commutativity in a semi-FTvN system; the following result and its proof are identical to those for FTvN systems (\cite{gowda-ftvn}, Section 2). 
\begin{proposition} \label{commutativity equivalence}
In a semi-FTvN system, the following are equivalent:  
\begin{itemize}
\item [$(a)$] $x$ and $y$ strongly commute, i.e.,  $\langle x,y\rangle=\langle \lambda(x),\lambda(y)\rangle.$
\item [$(b)$] $||\lambda(x)-\lambda(y)||=||x-y||.$
\item [$(c)$] $||\lambda(x+y)||= ||\lambda(x)+\lambda(y)||$.
\item [$(d)$] $\lambda(x+y)=\lambda(x)+\lambda(y).$
\end{itemize}
\end{proposition}

\gap

\begin{proof}
Using $(A1)$ (of Definition \ref{defn of semi-ftvn system}), we see that 
$$||\lambda(x)-\lambda(y)||^2-||x-y||^2=
2[\langle x,y\rangle-\langle \lambda(x),\lambda(y)\rangle]=||\lambda(x+y)||^2- ||\lambda(x)+\lambda(y)||^2.$$
The equivalence $(a)\Leftrightarrow (b)\Leftrightarrow (c)$ follows.\\
$(c)\Rightarrow (d)$: When $(c)$ holds, we have equality in the Cauchy-Schwarz inequality (\ref{CS}). Hence, 
one of the vectors, $\lambda(x+y)$ or $\lambda(x)+\lambda(y)$, is a nonnegative multiple of the other. Since their norms are equal (as we are assuming  $(c)$), these vectors are equal. Item $(d)$ follows. \\
Finally, $(d)\Rightarrow (c)$ is obvious.
\end{proof}

\gap

\noindent{\bf Remarks.} In a semi-FTvN system $(\V,\W,\lambda)$ we define the $\lambda$-orbit of an element $u\in \V$ by
$$[u]:=\{x\in \V: \lambda(x)=\lambda(u)\}.$$
One interesting consequence of Item $(b)$ in the above proposition is that
\begin{center}
    {\it if $x,y$ in $[u]$ strongly commute, then $x=y$. }
\end{center}

 We now state one of our main results.
 
\begin{theorem} \label{strong commutativity implies commutativity}
    In any semi-FTvN system $($in particular, in any FTvN system$)$, strong commutativity implies commutativity $($relative to any closed subgroup of $\mathrm{Aut(\V)})$.
\end{theorem}
\begin{proof} Consider a semi-FTvN system $(\V,\W,\lambda)$ and fix a closed subgroup ${\cal G}$ of $\mathrm{Aut(\V)}$.  Suppose $a$ and $b$ strongly commute, that is, $\langle a,b\rangle=\langle \lambda(a),\lambda(b)\rangle$. Fix any $D\in \mathrm{Lie({\cal G})}$ and consider the real-valued function
$$\phi(t):=\big\langle \exp(tD)a,b\big\rangle\quad (t\in \R).$$
Then, $\phi$ is differentiable and 
$$\phi(t)=\big\langle \exp(tD)a,b\big\rangle\leq \big\langle \lambda(\exp(tD)a),\lambda(b)\big \rangle=\big\langle \lambda(a),\lambda(b)\big\rangle=\langle a,b\rangle =\phi(0),$$
where the first inequality is due to the definition of a semi-FTvN system, the second equality is due to the definition of an automorphism, and the third equality is due to the strong commutativity of $a$ and $b$. So, $\phi$ attains its global maximum at $t=0$. From $\phi^\prime(0)=0$, we have $$\langle Da,b\rangle=\phi^\prime(0)=0.$$ Since $D$ is arbitrary in $\mathrm{Lie({\cal G})}$, $a$ and $b$ commute relative to ${\cal G}$. When ${\cal G}=\mathrm{Aut(\V)}$, we get the commutativity of $a$ and $b$. 
 \end{proof}

\subsection{The center and unit element}

The concepts of center and unit element have been defined and found to be useful in the setting of FTvN systems \cite{gowda-jeong-ftvn}. Now we extend these to semi-FTvN systems.

\begin{definition}\label{center and weak center}
    In a semi-FTvN system $(\V,\W,\lambda)$, we define the {\it center} ${\cal C}$ and {\it weak center} ${\cal C}_w$ as follows:
$${\cal C}:=\{x\in \V: x\,\,\mbox{strongly commutes with every}\,\,y\in \V\},$$
$${\cal C}_w:=\{x\in \V: x\,\,\mbox{  commutes with every}\,\,y\in \V\}=\{x\in \V: Dx=0\,\,\mbox{for all}\,\,D\in \mathrm{Lie(Aut(\V))}\}.$$
\end{definition}

Note: The second equality in ${\cal C}_w$ is due to the fact that $\langle Dx,y\rangle=0$ for all $y$ implies that $Dx=0$.
 
\begin{proposition} 
In a semi-FTvN system $(\V,\W,\lambda)$, 
the center ${\cal C}$ and the weak center ${\cal C}_w$  are closed (linear) subspaces of $\V$. Moreover, ${\cal C}\subseteq {\cal C}_w$ and $\lambda$ is linear on ${\cal C}$.
\end{proposition}
\begin{proof} Clearly, ${\cal C}_w$ is a closed subspace of $\V$. That ${\cal C}\subseteq {\cal C}_w$ follows from  Theorem \ref{strong commutativity implies commutativity}.
We now show that ${\cal C}$ is a closed subspace of $\V$ and that $\lambda$ is linear on it.
Suppose $a,b\in {\cal C}$. Then, by  Proposition \ref{elementary properties of lambda}(a), any nonnegative multiple of $a$ is also in ${\cal C}$. Moreover, as $a$ strongly commutes with $-a$, by Proposition \ref{commutativity equivalence}(d), $\lambda(-a)=-\lambda(a)$. It follows that any multiple of $a$ is in ${\cal C}$. Now repeated use of Item (d) in Proposition \ref{commutativity equivalence} shows that $a+b\in {\cal C}$. That ${\cal C}$ is closed comes from the continuity of $\lambda$ and Item $(a)$ in Proposition \ref{commutativity equivalence}. Finally,  Proposition \ref{commutativity equivalence}(d)  and positive homogeneity of $\lambda$ gives the linearity on ${\cal C}$.
\end{proof}

\smallgap

\begin{definition} In the setting of a  semi-FTvN system $(\V,\W,\lambda)$,  a nonzero element $e\in \V$ is said to be 
     a unit  if  ${\cal C}=\R\,e$.
\end{definition}

\gap

\noindent{\bf Remarks.} In an FTvN system, the center can be described in many ways, see Section 6 in \cite{gowda-jeong-ftvn}. In particular, it is shown in \cite{gowda-jeong-ftvn}, Proposition 6.3, that $u\in {\cal C}$ if and only if $[u]=\{u\}.$
A result of this type is proved below for a semi-FTvN system.
\gap

\begin{proposition}
Let $(\V,\W,\lambda)$ be a semi-FTvN system. Consider the following statements:
\begin{itemize}
    \item [$(i)$] $c\in {\cal C}$.
    \item [$(ii)$] $[c]=\{c\}$.
    \item [$(iii)$] $Ac=c$ for all $A\in \Aut(\V)$.
    \item [$(iv)$] $c\in {\cal C}_w$.
    \end{itemize}
    Then, $(i)\Rightarrow (ii)\Rightarrow (iii)\Rightarrow (iv)$.
    In particular, if $e$ is a unit element, then $Ae=e$ for all $A\in \Aut(\V,\W,\lambda)$.
\end{proposition}

\noindent{Note:} Example \ref{isolated lambda orbit} below shows that the (reverse) implication $(ii)\Rightarrow (i)$ need not hold. Also, Example \ref{usual rn} shows that the (reverse) implication $(iv)\Rightarrow (iii)$ need not hold (even in an FTvN system). 
\smallgap

\begin{proof}
$(i)\Rightarrow (ii)$: Assume $(i)$,  fix $c\in {\cal C}$ and $x\in \V$ such that  $\lambda(x)=\lambda(c)$. As $c$ strongly commutes with $x$, from Proposition 3.5(b), $||x-c||=||\lambda(x)-\lambda(c)||=0$. So, $x=c$. This gives $(ii)$.\\
$(ii)\Rightarrow (iii)$: Assume $(ii)$ and let $A\in \Aut(\V)$. Then,  $Ac\in [c]=\{c\}$, so $Ac=c$.\\
$(iii)\Rightarrow (iv)$: Assume $Ac=c$ for all $A\in \Aut(\V)$. Then, for any $D\in \Lie(\Aut(\V))$ and $t\in \R$, we have  $\exp(tD)\in \Aut(\V)$, hence  $\exp(tD)c=c$. Now, differentiation leads to $Dc=0$. By Definition \ref{center and weak center}, $c\in {\cal C}_w$.
\end{proof}

\begin{definition}
In a semi-FTvN system $(\V,\W,\lambda)$, a set $E\subseteq \V$ is said to be a
\begin{itemize}
\item a spectral set if $x\in E\Rightarrow [x]\subseteq E$ or, equivalently, $E=\lambda^{-1}(Q)$ for some $Q\subseteq \W$;
\item a weakly spectral set if $A(E)\subseteq E$ for all $A\in \Aut(\V,\W,\lambda)$.
\end{itemize}
\end{definition}

\smallgap

\begin{proposition}
In a semi-FTvN system, the center is a spectral set and the weak center is a weakly spectral set.
\end{proposition}

\begin{proof}
To simplify the notation, let ${\cal G}:=\Aut(\V)$.\\
If $x\in {\cal C}$, then by the above proposition, $[x]= \{x\}\subseteq {\cal C}$. Hence, the center is a spectral set.\\
 To show that ${\cal C}_w$ is weakly spectral, we need to show that $A({\cal C}_w)\subseteq {\cal C}_w$ for all $A\in \G$. In view of Definition \ref{center and weak center}, we need to show: 
\begin{center}
    for any $c\in {\cal C}_w$ $A\in {\cal G}$, and $D\in \Lie({\cal G})$, it holds that
$ D(Ac)=0$.
\end{center}
We fix $c\in {\cal C}_w$,  $A\in {\cal G}$, and $D\in \Lie({\cal G})$. Then, for any $t\in \R$, $exp(tD)\in {\cal G}$,
$$exp(tA^{-1}DA)=A^{-1}exp(tD)A\in {\cal G}$$ and so, $A^{-1}DA\in \Lie({\cal G})$. As $c\in {\cal C}_w$, we have $A^{-1}DAc=0$ and $DAc=0$. As $D$ is arbitrary in $\Lie({\cal G})$, $Ac\in {\cal C}_w$.
This completes the proof.
\end{proof}

We now describe some examples. Recall that in the setting of  a semi-FTvN system $(\V,\W,\lambda)$, $\mathrm{Aut(\V)}$ is an abbreviation for 
$\mathrm{Aut(\V,\W,\lambda)}$.

\begin{example}\label{absolute value example}
Consider the semi-FTvN system $(\R,\R,\lambda)$, where $\lambda(x)=|x|$.
As\\ $\mathrm{Aut(\R)}=\{I,-I\}$ and $\mathrm{Lie(\Aut(\R))}=\{0\}$, where $I$ and $0$ denote the Identity and zero transformations respectively, we see that any two elements $a,b\in \R$  commute. On the other hand,  $a$ and $b$ strongly commute (that is, $ab=|a|\,|b|$) if and only if one of them (either $a$ or $b$) is a nonnegative multiple of the other. Thus, in this system,
    ${\cal C}=\{0\}$ while ${\cal C}_w=\R$. 
    
\end{example}

\smallgap

\begin{example}
 Consider the semi-FTvN system $(\V,\R,\lambda)$, where $\V$  is a  real finite dimensional inner product space with $dim(\V)>1$ and  $\lambda(x)=||x||.$  By isomorphism considerations, we may assume that $\V=\R^n$ ($n>1$) with the usual inner product. Consider $a,b\in \R^n$. We claim that {\it  $a$ and $b$  commute in the semi-FTvN system $(\R^n,\R,\lambda)$ if and only if $b$ is a scalar multiple of $a$ or $a$ is a scalar multiple of $b$}. To see this, suppose without loss of generality, $a\neq 0$.
  As  $\mathrm{Aut(\V, \R,\lambda)}={\cal O}(n)$ (the orthogonal group of $\R^n$) and $n>1$,  the corresponding  Lie algebra consists of all skew-symmetric matrices (\cite{hall}, 2.5.5); hence,  $a$ and $b$  commute if and only if $\langle Da,b\rangle=0$ for every skew-symmetric matrix $D$. It is known that the orthogonal complement of the space of the skew-symmetric matrices is the space of the symmetric matrices relative to the trace inner product $\langle A,B\rangle_{\tr}={\tr}(AB^\top)$. Since $\langle Da,b\rangle=\langle D,ba^\top\rangle_{\tr}$, we see that $a$ and $b$ commute if and only if $ba^\top$ is symmetric. Then, $ba^\top=ab^\top$, hence $ba^\top a=ab^\top a$. Since $a\neq 0$, we have $b=\frac{b^\top a}{\Vert a\Vert^2}a$. Thus, $b$ is a scalar multiple of $a$.  Conversely, if $b$ is a scalar multiple of $a$, then $\langle Da,b\rangle=0$ for every skew-symmetric matrix $D$, hence $a$ and $b$  commute. This proves our claim. Now suppose $a\in {\cal C}_w$. Let $\one$ be the vector of all $1$s and $d$ be a vector with all entries distinct (recall that $n>1$). Since $a$  commutes with both $\one$ and $d$, we must have  $a=0$. Hence  ${\cal C}_w=\{0\}$. 
  We remark that $a$ and $b$ strongly commute in this system if and only if  $\langle a,b\rangle=||a||\,||b||$ in which case, $a$ (or $b$) is a nonnegative multiple of $b$ (respectively, $a$). Consequently, since $-a$ is not a nonnegative multiple of $a$ unless $a=0$, we see that
  (the center) ${\cal C}=\{0\}$. Thus, in this system, ${\cal C}={\cal C}_w=\{0\}$. {\it Note that this is different from the $n=1$ case (of Example   \ref{absolute value example}).}
  \end{example}

\smallgap

\begin{example} \label{usual rn}
Consider $\R^n$ with the usual inner product. For any $x\in \R^n$, let $x^\downarrow$ denote the vector obtained by rearranging the entries of $x$ in the decreasing order. As $\langle x,y\rangle \leq \langle x^\downarrow,y^\downarrow\rangle$ (which is the Hardy-Littlewood-Polya inequality) and $||x^\downarrow||=||x||$, we see that $(\R^n,\R^n,\lambda)$ is a semi-FTvN system, where $\lambda(x)=x^\downarrow$. (Actually, this is an FTvN system arising from the Euclidean Jordan algebra  $(\R^n, \langle \cdot,\cdot\rangle, \circ)$, where, for $x,y\in \R^n$, $\langle x,y\rangle$ is the usual inner product and $x\circ y$ is the componentwise product.) As $\mathrm{Aut}(\R^n,\R^n,\lambda)$ is the group of all permutation matrices on $\R^n$, we see that the corresponding Lie algebra is $\{0\}$. Hence, any two elements in this semi-FTvN system commute.  With ${\bf 1}$ denoting the vector of ones in $\R^n$, we shall see (in Theorem \ref{weak equals operator commutativity}) that
    ${\cal C}=\R{\bf 1}$ while ${\cal C}_w=\R^n$. 
    We note that when $n>1$, the vectors $a=(1,0,\ldots,0)$ and $b=(0,1,0,\ldots, 0)$  commute, but not strongly:  $\langle a,b\rangle=0$, while $\langle \lambda(a),\lambda(b)\rangle =1$.
\end{example}

\begin{example}\label{isolated lambda orbit}
    In $\R^2$, let $\V:=\{r(1,2):r\in \R\}.$ Consider $\lambda:\V\rightarrow \R^2$ with $\lambda(x)=x^\downarrow.$ Then, $(\V,\R^2,\lambda)$ is a semi-FTvN system, where for any $x\in \V$, $[x]=\{y\in \V:\lambda(x)=\lambda(y)\}=\{x\}$. It follows that (\ref{ftvn defn}) fails and so, $(\V,\R^2,\lambda)$ is not an FTvN system.  Morover, $\mathrm{Aut(\V)}=\{I\}$, ${\cal C}=\{0\}$, and ${\cal C}_w=\V$.
\end{example}

{\it Note}: It will be shown in the more general setting of a Euclidean Jordan algebra that commutativity is the same as operator commutativity. Consequently, in the Euclidean Jordan algebra ${\cal S}^n$ of all $n\times n$ real symmetric matrices, commutativity is the same as the usual commutativity of matrices.

\gap

\begin{example} \label{NDS} ({\it Normal decomposition system} \cite{lewis}) Let $\V$ be a real inner product space, $\G$ be a closed subgroup of the orthogonal group of $\V$, and $\gamma : \V  \to \V$ be a map satisfying the following conditions:
	\begin{itemize}
		\item [$(a)$] $\gamma$ is $\G$-invariant, that is, $\gamma(Ax) = \gamma(x)$ for all $x \in \V$ and $A \in \G$.
		\item [$(b)$] For each $x\in \V$, there exists $A\in \G$ such that $x=A\gamma(x)$.
		\item [$(c)$] For all $x,y\in \V$, we have $\ip{x}{y} \leq \ip{\gamma(x)}{\gamma(y)}$.
	\end{itemize}
	Then, $(\V, \G, \gamma)$ is called a {\it normal decomposition system}.
In this setting, we have (cf. \cite{lewis}, Proposition 2.3 and  Theorem 2.2):
 For any two elements $x$ and $y$ in $\V$, 
		\[ \max_{A \in \G}\, \ip{Ax}{y} = \ip{\gamma(x)}{\gamma(y)} \]
		and $\ip{x}{y} = \ip{\gamma(x)}{\gamma(y)}$  if and only if there exists an $A \in \G$ such that $x = A\gamma(x)$ and $y = A\gamma(y)$.

As observed in \cite{gowda-ftvn}, if $(\V, \G, \gamma)$ is a  normal decomposition system, then  $(\V, \V, \gamma)$  becomes a FTvN system. We refer to \cite{gowda-jeong-ftvn}, Theorem 7.3, for a characterization of normal decomposition systems among FTvN systems.   From the above result, we see a straightforward description of strong commutativity. For commutativity, we observe that ${\cal G}$ is a closed subgroup of the (possibly larger) group  $\mathrm{Aut(\V,\V,\gamma)}$. Since the group ${\cal G}$ is readily available in this setting, it may be advantageous to work with `commutativity relative to ${\cal G}$'. 
As in the proof of Theorem \ref{strong commutativity implies commutativity}, we see that strong commutativity implies commutativity relative to ${\cal G}$.
\end{example}

\begin{example}
Let $M_{m,n}(\C)$ denote the space of all $m\times n$ 
complex matrices under the inner product 
$\langle X,Y\rangle={\rm Re}\,\tr(XY^*)$, where $Y^*$ denotes the conjugate-transpose of $X$. With $\lambda(X):= \Diag(s(X))$ denoting the (rectangular) diagonal matrix of the singular values of $X$ written in decreasing order, we recall von Neumann's result: For $X, Y\in M_{m,n}(\C)$,
$$\langle X,Y\rangle\leq \langle \lambda(X),\lambda(Y)\rangle$$ with equality if and only if there exist unitary matrices $U$ and $V$ such that 
$X=U\, \Diag(s(X))\,V$ and $Y=U\, \Diag(s(Y))\,V.$ Using the above inequality, for any given $A,B\in M_{m,n}(\C)$ with $A=U\, \Diag(s(A))\,V$, we can define $X:=U\, \Diag(s(B))V$ to verify that 
\begin{equation} \label{M_{m,n} is ftvn }
	\max \Big\{\! \ip{A}{X} : \, X \in [B] \Big\} = \langle \lambda(A),\lambda(B)\rangle \quad (\forall\, A, B \in \V).
\end{equation}
Thus, $\big (M_{m,n}(\C),M_{m,n}(\C),\lambda\big )$ is a FTvN system. 
In this system,  von Neumann's result characterizes strong commutativity. In particular, we have the following:

\begin{center}
 {\it    If $X,Y\in M_{m,n}(\C)$ and
    $\langle X,Y\rangle =\langle \lambda(X),\lambda(Y)\rangle$, then $XY^*\,\,\mbox{and}\,\,X^*Y$ are Hermitian.
}
\end{center}
By replacing complex numbers by real numbers, we can consider the space $M_{m,n}(\R)$ and obtain analogous results. In particular, we see that $\big (M_{m,n}(\R),M_{m,n}(\R),\lambda\big )$ is an FTvN system. (Note: The unitary matrices become  orthogonal matrices, see \cite{horn-johnson}, Theorem 7.3.5 or \cite{horn-johnson-topics}, Theorem 3.1.1,  and  \cite{lewis-sendov}, Theorem 4.6.) 
Moreover, 
\begin{center}
 {\it    If $X,Y\in M_{m,n}(\R)$ and
    $\langle X,Y\rangle =\langle \lambda(X),\lambda(Y)\rangle$, then $XY^T\,\,\mbox{and}\,\,X^TY$ are symmetric.
}
\end{center}

We recall, see Corollary \ref{commutativity in Mn}, that in the particular setting of $M_{n,n}(\R)$ (which is $M_n$ as per our earlier notation),
the condition that $XY^T\,\,\mbox{and}\,\, X^TY$ are symmetric
describes commutativity relative to a group of transformations. Results analogous to this seem to hold for $M_{m,n}(\R)$ and $M_{m,n}(\C)$, see \cite{lewis-sendov}.
\\

We draw attention to the following characterization result due to Lewis and Sendov, see \cite{lewis-sendov}, Lemma 4.3. We state the result (only) for $M_n:=M_{n,n}(\R)$. Recall that a signed permutation is a square matrix where each row and column has exactly one nonzero entry which is $\pm 1$.

\begin{theorem}
Let  $X,Y\in M_{n}$. Then $XY^T$ and $X^TY$ are symmetric if and only if there is a signed permutation $Q$ and orthogonal matrices $U$ and $V$ such that 
$X=U \Diag(Qs(X))V$ and $Y=U \Diag(s(Y))V.$
\end{theorem}

We end this example by describing the center ${\cal C}$ and weak center ${\cal C}_w$ of $M_n$.  
If $C\in {\cal C}$,  then $\langle C,X\rangle =\langle \lambda(C),\lambda(X)\rangle$ for all $X\in M_n(\R).$ In particular, $\langle C,-C\rangle =\langle \lambda(C),\lambda(-C)\rangle$. As $\lambda(C)=\lambda(-C)$, we get $\langle C,-C\rangle\geq 0$ and, consequently, $C=0$. Thus, ${\cal C}=\{0\}$.
\\
Now consider the case $n=1$. Then, $M_1=\R$ and $\lambda(x)=|x|$. From Example \ref{absolute value example}, ${\cal C}_w=\R.$ On the other hand, suppose $n>1$ and $X\in {\cal C}_w$. Then, for every $Y\in M_n$, $\langle AX+XB,Y\rangle=0$, where $A$ and $B$ are arbitrary skew-symmetric matrices. (See the proof of Item $(ii)$ in Theorem 2.9 with ${\cal G}_1={\cal G}_2={\cal O}(n)$.) We see that $AX=0$ for every skew-symmetric matrix $A$,
consequently, $X=0$. Thus, when $n>1$, ${\cal C}_w=\{0\}$.

\end{example}
\section{The FTvN system induced by a Euclidean Jordan algebra}

     Suppose $(\V, \langle\cdot,\cdot\rangle, \circ)$ is a Euclidean Jordan algebra of rank $n$  with unit element $e$ \cite{faraut-koranyi} . Here, for any two elements $x$ and $y$, $\langle x,y\rangle$ denotes the inner product and $x\circ y$ denotes the Jordan product. For any $x\in \V$, by the spectral decomposition theorem (\cite{faraut-koranyi}, Theorem III.1.2), we can write
     \begin{equation}\label{spectral decomposition}
     x=x_1e_1+x_2e_2+\cdots +x_ne_n,
     \end{equation}
     where $\{e_1,e_2,\ldots, e_n\}$ is a Jordan frame and $x_1,x_2,\ldots, x_n$ are (unique) real numbers (called the eigenvalues of $x$). We define the trace of $x$ as  $\tr(x):=x_1+x_2+\cdots+x_n$ and trace inner product over $\V$ by $\langle x,y\rangle_{\tr}: =\tr(x\circ y)$. This inner product is compatible with the Jordan product. Moreover, Jordan frames, eigenvalues and spectral decompositions remain the same when we replace the given inner product with the trace inner product. {\it Henceforth, in the setting of a Euclidean Jordan algebra, we assume that $\V$ carries this trace inner product.} \\Let 
     $\Aut(\V)$ denote the group of all invertible linear transformations on $A$ satisfying the condition 
     $$A(x\circ y)=Ax\circ Ay\,\,\,\mbox{for all}\,\,x,y\in \V.$$
     Elements of  $\Aut(\V)$ are called (algebra) automorphisms of $\V$. (These are different from the `cone' automorphisms, namely, elements of $\Aut(K)$, where $K$ is the symmetric cone of $\V$.)\\
     Given the spectral decomposition of $x$ (as above), let $\lambda(x)$ in $\R^n$ denote the vector of eigenvalues of $x$ written in the decreasing order, that is,
     $$\lambda(x)=(x_1,x_2,\ldots, x_n)^\downarrow.$$
     Then, it is shown in \cite{gowda-jeong-ftvn}, Section 4, Example 4.6, that $(\V,\R^n, \lambda)$ is an FTvN system. In this setting, it is known, see \cite{gowda-jeong-ftvn}, Section 7, Example 4.6, that 
     \begin{center}
         {\it $A\in \mathrm{Aut(\V,\R^n,\lambda)}$ if and only if $A\in \mathrm{Aut(\V)}.$} 
         \end{center}
     (So, in the setting of a Euclidean Jordan algebra, the automorphisms of the FTvN system $(\V,\R^n,\lambda)$ are the same as the algebra automorphisms.)\\
     We recall that in a Euclidean Jordan algebra, two elements $a$ and $b$ {\it operator commute} if $L_aL_b=L_bL_a$ (where $L_a$ is the linear operator defined by $L_a(x):=a\circ x$ for all $x\in \V$), or equivalently, $a$ and $b$ have their spectral representations with respect to the same Jordan frame (\cite{faraut-koranyi}, Lemma X.2.2). We  say that $a$ and $b$ {\it strongly operator commute} if there exists a Jordan frame $\{e_1,e_2,\ldots, e_n\}$ such that  $a$ and $b$ have spectral decompositions
     $$a=a_1e_1+a_2e_2+\cdots+a_ne_n\quad\mbox{and}\quad b=b_1e_1+b_2e_2+\cdots+b_ne_n$$
     with $a_1\geq a_2\geq \cdots\geq a_n$ and $b_1\geq b_2\geq \cdots\geq b_n$. It is shown in \cite{gowda-ftvn}, Proposition 4.4, that {\it in the induced FTvN system $(\V,\R^n,\lambda)$ $a$ and $b$ strongly commute if and only if they strongly operator commute in the Euclidean Jordan algebra $\V$.}
     We now prove a similar result for  commutativity.
     
     \begin{theorem}\label{ commutativity equals operator commutativity}
         Consider a  Euclidean Jordan algebra $(\V, \langle\cdot,\cdot\rangle, \circ)$ of rank $n$ with  unit element $e$.  Let $a,b\in \V$. Then,  $a$ and $b$  commute in the FTvN system $(\V,\R^n,\lambda)$ if and only if they operator commute in the Euclidean Jordan algebra $\V$.
         \end{theorem}
        \begin{proof} We have noted previously that $(\V,\R^n,\lambda)$ is a FTvN system. Furthermore,  as noted above,
    $\mathrm{Aut(\V,\R^n,\lambda)}=\mathrm{Aut(\V)}$. It is known (\cite{koecher}, Lemma 7 and Theorem 8) that a transformation $D\in \mathrm{Lie(\Aut(\V))}$ (called a derivation) if and only if it is a sum of transformations of the form $L_uL_v-L_vL_u$. So, $a$ and $b$ commute if and only if $\langle (L_uL_v-L_vL_u)a,b\rangle =0$ for all $u,v\in \V$. Since the latter statement simplifies to $\langle (L_aL_b-L_bL_a)u,v\rangle=0$, and $u,v$ are arbitrary, commutativity of $a$ and $b$ reduces to the equality  $L_aL_b-L_bL_a=0$, which is the operator commutativity of $a$ and $b$.
 
\end{proof}

 The following result provides a description of the center and weak center in a Euclidean Jordan algebra. Recall that the unit element $e$ in the algebra $\V$ satisfies the condition $x\circ e=x$ for all $x$.
 
 \begin{theorem}\label{weak equals operator commutativity}
      Suppose $\V$ is a Euclidean Jordan algebra with unit element $e$. Let ${\cal C}$ and ${\cal C}_w$ denote, respectively, the center and weak center in $\V$. Then the following statements hold: 
     \begin{itemize}  
     \item [$(i)$] ${\cal C}=\R\,e$;
         \item [$(ii)$] If $\V$ is simple, then ${\cal C}_w=\R\,e$;
         \item [$(iii)$] 
         If $\V=\Pi_{i=1}^{k}\V_i$, where $\V_i$ is a simple Euclidean Jordan algebra with unit $e_i$, then 
         ${\cal C}_w=\Pi_{i=1}^{k}\,\R\,e_i.$
         In particular, when  $\V=\R^n$ with componentwise product as Jordan product, we have ${\cal C}_w=\R^n$.
         \end{itemize}
         \end{theorem}
         \begin{proof}
         Item $(i)$ is known, see \cite{gowda-jeong-ftvn}, Section 6, Example 4.6.\\
         $(ii)$ Clearly, $\R\,e\subseteq {\cal C}_w$. To see the reverse inclusion, suppose that $\V$ is simple and $a\in {\cal C}_w$.  By the previous theorem, $a$ operator commutes with every $x\in \V$. In particular, $a$ operator commutes with every primitive idempotent $c$ in $\V$. Writing the spectral decompositions of $a$ and (the primitive idempotent) $c$ with respect to a common Jordan frame, we see that (every) $c$ must appear in some spectral decomposition of $a$. 
          Suppose, to get a contradiction, $a$ is not a multiple of $e$, that is, $a$ has two distinct eigenvalues. Let $a_1,a_2,\ldots, a_k$ $(k>1)$, be the distinct eigenvalues of $a$. Then, the closed sets
         $$E_i:=\{x\in \V: \langle a,x\rangle=a_i\}\,\,(1\leq i\leq  k)$$
         are disjoint. Since every primitive idempotent appears in some spectral decomposition of $a$ and $\V$ carries the trace inner product, we see that every primitive idempotent is in some $E_i$. This means that the set ${\cal T}(\V)$ of all primitive idempotents in $\V$ is contained in the disjoint union of closed sets $E_i$. However, ${\cal T}(\V)$ is connected in (any) simple algebra $\V$, see \cite{faraut-koranyi}, Page 71 or \cite{gowda-jeong-connected}, Theorem 3.1. We reach a contradiction. Thus, $a$ must be a multiple of $e$, proving our assertion.\\
         $(iii)$ Suppose $\V=\Pi_{i=1}^{k}\V_i$, where each $\V_i$ is simple. For $a,b\in \V$, we write $a=(a_1,a_2,\ldots, a_k)$ and $b=(b_1,b_2,\ldots, b_k)$ and note that $\langle a,b\rangle =\sum_{i=1}^{k}a_ib_i$ and $a\circ b=(a_1\circ b_2,a_2\circ b_2,\ldots, a_k\circ b_k)$. Then $a$ and $b$ operator commute if and only if $a_i$ and $b_i$ operator commute in $\V_i$ for each $i$. It follows that $a$ is in the weak center of $\V$ if and only if $a_i$ is in the weak center of $\V_i$ for each $i$. As $\V_i$ is simple, from $(ii)$, the weak center of $\V_i$ is $\R\,e_i$, where $e_i$ is the unit element of $\V_i$. We see that ${\cal C}_w$ is the product of $\R\,e_i$, $i=1,2,\ldots, k$. In particular, when $\V=\R^n$, $\V_i=\R$ with $e_i=1$. Hence, in this setting, the weak center is  $\R^n$.
          
        \end{proof}
\noindent{\bf Remarks.} As pointed out by Michael Orlitzky (in a private communication), Item $(ii)$ in the above theorem can also be seen by mimicking the proof of Lemma VIII.5.1 in \cite{faraut-koranyi}.

\gap

        We have observed before that the composition of two semi-FTvN systems is again a semi-FTvN system. The following example shows that a similar statement fails for FTvN systems, that is, {\it  composition of FTvN systems need not be a FTvN system.}

\gap
        
\begin{example}\label{composition example}
Let $\widetilde{\R^2}$ denote the inner product space consisting on the vector space $\R^2$ with the inner product $\langle x,y\rangle_*:=2\langle x,y\rangle$. (Recall that $\R^2$ carries the usual inner product.) Consider two FTvN systems $(\widetilde{\R^2},\widetilde{\R^2},\lambda)$ and $(\widetilde{\R^2},\R^2,\mu)$, where for any $x=(x_1,x_2)$, we define $\lambda(x):=x^\downarrow$ and $\mu(x):=(x_1+|x_2|,x_1-|x_2|).$ In fact, $(\widetilde{\R^2},\widetilde{\R^2},\lambda)$ is the scaled version (on $\R^2$) of the FTvN system of Example \ref{usual rn} and $(\widetilde{\R^2},\R^2,\mu)$ is the scaled version of the Jordan spin algebra on $\R^2$, see Example 4.14 in \cite{gowda-jeong-ftvn}. Let $\nu:=\mu\circ \lambda$. Clearly, $(\widetilde{\R^2},\R^2,\nu)$ is a semi-FTvN system. We claim that $(\widetilde{\R^2},\R^2,\nu)$ is not a FTvN system. Consider the vectors 
$c=(1,-1)$ and $u=(-1,2)$ so that $\nu(c)=(2,0)$ and $\nu(u)=(3,1)$. Suppose $(\widetilde{\R^2},\R^2,\nu)$ is a FTvN system  so that condition 
      (\ref{ftvn defn}) holds. Then, there is an $x=(x_1,x_2)$ such that $\nu(x)=\nu(u)$ and $\langle c,x\rangle_*=\langle \nu(c),\nu(u)\rangle$. Writing $\lambda(x)=x^\downarrow=(x_1^\downarrow,x_2^\downarrow)$, these translate to $(x_1^\downarrow+|x_2^\downarrow|,x_1^\downarrow-|x_2^\downarrow|)=(3,1)$ and $x_1-x_2=3.$ As $x_1^\downarrow =2$ and $|x_2^\downarrow|=1$, these conditions cannot hold simultaneously. Thus, (\ref{ftvn defn}) fails to hold for $\nu$. 
\end{example}

\section{The semi-FTvN system induced by a complete hyperbolic polynomial}\label{hyperbolic polynomials}
Let $\V$ be a nonzero finite-dimensional real vector space, $0\neq e\in \V$, and $p$ be a real homogeneous polynomial
of degree $n$ on $\V$. We say that {\it $p$ is hyperbolic relative to $e$} if $p(e)\neq 0$ and
for every $x\in \V$,   the roots of the univariate polynomial $t\rightarrow p(te-x)$ are all 
real. Given such a polynomial $p$,  for any $x\in \V$, let $\lambda(x)$  denote the vector
of roots of the polynomial $t\rightarrow p(te-x)$  with entries written in decreasing
order (so, $\lambda(x)\in \R^n$). We shall call $\lambda:\V\rightarrow \R^n$ the {\it eigenvalue map} of $p$ (relative to $e$).\\

Without imposing additional conditions on $p$, we list some known properties:
\begin{itemize}
    \item [(P1)] Proposition \ref{elementary properties of lambda} continues to hold for the eigenvalue map of $p$ (though the proofs are not elementary), see \cite{bauschke et al}. 
\item [(P2)] 
In the canonical setting of $\V=\R^n$ with $p(e)=1$, for any two elements
$x,y\in \R^n$,  there exist real symmetric $n\times n$ matrices $A$ and $B$ such that 
\begin{equation}\label{gurvitz result}
\lambda(tx+sy)=\lambda(tA+sB)\quad (t,s\in \R),
\end{equation}
 where the right-hand side denotes the eigenvalue vector
of a symmetric matrix. This is an observation due to  Gurvits \cite{gurvitz} based on the validity of Lax conjecture (see \cite{lewis et al, helton-vinnikov}). \\One important consequence of (\ref{gurvitz result}) is the validity of Lidskii inequality (from a similar inequality for real symmetric matrices):
\begin{equation}\label{lidskii inequality}\lambda(x)-\lambda(y)\prec \lambda(x-y)\,\,\mbox{for all}\,\,x\\,y\in \V. 
\end{equation}
\item [(P3)] 
The set
$K:=\{x\in \V: \lambda(x)\geq 0 \}$
is a closed convex cone with nonempty interior, called the {\it hyperbolicity cone} corresponding to  $p$ and $e$. For any $d\in K^\circ$ (the interior of $K$), $p$ is hyperbolic relative to $d$ and the corresponding hyperbolicity cone is (still) $K$. In fact, $K^\circ$ is the connected component of the set $\{x\in \V:p(x)\neq 0\}$ that contains $e$. These results are shown in \cite{garding}. Moreover, from \cite{bauschke et al}, Fact. 2.9, we have
\begin{equation}\label{lineality space}
K\cap -K=\{x\in \V: \lambda(x)=0\}.
\end{equation}
Some deeper results on hyperbolicity cones include:
$K$ is facially exposed \cite{renegar hyperbolic program}, and amenable \cite{Lourenco-hyperbolic cone amenable}.
\end{itemize}

\begin{definition} (\cite{bauschke et al}, Definition 2.8)
   Let $p$ be hyperbolic on  $\V$ relative to $e$. Let $\lambda$ be the corresponding eigenvalue map. We say that  $p$ is complete (relative to $e$) if  the following implication holds:
   $$\lambda(x)=0\Rightarrow x=0.$$ 
\end{definition}

\noindent{\bf Remarks.} Let $p$ (hyperbolic on  $\V$ relative to $e$) be complete.  In view of (\ref{lineality space}), this completeness property of $p$ is equivalent to (the corresponding hyperbolicity cone) $K$ being pointed (meaning $K\cap -K=\{0\}$). Since $K$ is the hyperbolicity cone of $p$ relative to any $d\in K^\circ$, we see that the  $p$ is complete relative to any $d$ in $K^\circ.$ 

\gap
A few general ways of constructing complete hyperbolic polynomials are as follows: 
 \begin{itemize}
 \item Suppose $p$ and $q$ are hyperbolic on $\V$ relative to the same $e$. If $q$ is complete, then so is the product $pq$.
     \item Given a finite set of complete hyperbolic polynomials over $\V$ (all hyperbolic with respect to the same $e$), their product is again complete and hyperbolic. 
\item Suppose $p$ is a complete hyperbolic polynomial of degree $n$ over $\V$ with respect to $e\in \V$; let $\lambda:\V\rightarrow \R^n$ denote the corresponding eigenvalue map.   Let $q$ be a permutation invariant complete hyperbolic polynomial of degree $m$ on $\R^n$ with respect to $d=(1,1,\ldots,1)\in \R^m$; let $\mu$ denote the corresponding eigenvalue map. Then, the composition $q\circ \lambda$ is a complete hyperbolic polynomial of degree $m$ on $\V$ with respect to $e$ whose eigenvalue map is $\mu\circ \lambda$. (That $q\circ \lambda$ is hyperbolic with eigenvalue map $\mu\circ \lambda$ is proved in \cite{bauschke et al}, Theorem 3.1.) 
As $(\mu\circ \lambda)(x)=0\Rightarrow \mu(\lambda(x))=0\Rightarrow \lambda(x)=0\Rightarrow x=0$, we see the completeness of $q\circ \lambda$.
\end{itemize}

\begin{theorem}\label{construction}
Suppose $p$ is a complete hyperbolic polynomial on $\V$ relative to $e$. Let $\lambda$ be the corresponding eigenvalue map. Define
\begin{equation}\label{polarization identity}
    \langle x,y\rangle:= \frac{1}{4}\Big \{
||\lambda(x+y)||^2-||\lambda(x-y)||^2\Big \}\quad (x,y\in \V)
\end{equation}
(where the right-hand side is computed in $\R^n$ equipped with the usual norm). Then, relative to the above, $\V$ becomes a (real, finite-dimensional) inner product space and the triple $(\V,\R^n,\lambda)$ becomes a semi-FTvN system. Moreover, the hyperbolicity cone $\{x\in \V:\lambda(x)\geq 0\}$ is a proper cone.
\end{theorem}
\begin{proof}
That $\langle x,y\rangle$ defines  an inner product on $\V$,
$\lambda:\V\rightarrow \R^n$ is norm-preserving, and 
$\langle x,y\rangle\leq \langle\lambda(x),\lambda(y)\rangle$ for all $x,y\in \V$, are shown in   \cite{bauschke et al}, Theorem 4.2 and Proposition 4.4.
    Additionally, the pointedness of the hyperbolicity cone comes from  (\ref{lineality space}). As $K$ is a pointed closed convex cone with nonempty interior, by definition, it is a proper cone.
    \end{proof}

\noindent{\bf Remarks.} Suppose $p$ is a complete hyperbolic polynomial on $\V$ and consider the induced semi-FTvN system $(\V,\R^n,\lambda)$ defined above. Then, following
\cite{bauschke et al}, Definition 5.1, we say that $p$ is
{\it isometric} if for all $y,z\in \V$, there exists $x\in \V$ such that $$\lambda(x)=\lambda(z)\quad \mbox{and}\quad \lambda(x+y)=\lambda(x)+\lambda(y).$$ Then,  Proposition 5.3 in \cite{bauschke et al} shows that  when $p$ is complete and isometric, $(\V,\R^n\lambda)$ satisfies (\ref{ftvn defn}), hence a FTvN system. In summary,
\\
{\it Every complete hyperbolic polynomial induces a semi-FTvN system and every complete isometric hyperbolic polynomial induces an FTvN system.} 

\gap

We list below a few examples:

\gap

\begin{example}
On $\R^n$ with $x=(x_1,x_2,\ldots, x_n)$,
 $p(x):=x_1x_2\cdots x_n$ is hyperbolic with respect to $e=(1,1,\ldots,1)$. Here $\lambda(x)=x^\downarrow$. This $p$ is complete. The semi-FTvN system reduces to the one given in  Example \ref{usual rn}.
 \end{example}
 
  \begin{example}  
On $\R^n$, $n>1$,  $q(x):=x_1^2-(x_2^2+x_3^2+\cdots+x_n^2)$ is hyperbolic with respect to $e=(1,0,\ldots, 0).$ This $q$ is complete. The induced semi-FTvN corresponds to the Jordan spin algebra ${\cal L}^n$ that appears in the study of Euclidean Jordan algebras.
\end{example}

\begin{example}
Let $\V$ be an Euclidean Jordan algebra of rank $n$ and unit element $e$. Then, $p(x):=\det(x)$ is hyperbolic with respect to $e$. In this case, the inner product defined in (\ref{polarization identity}) reduces to the trace inner product.
\end{example}

\gap

For the rest of this section, we assume that {\it $p$ is a complete hyperbolic polynomial on $\V$ relative to $e$ with eigenvalue map $\lambda$, $K$ is the corresponding hyperbolicity cone, and $(\V,\R^n,\lambda)$ is the induced semi-FTvN system.} 

\begin{definition}
A linear transformation $A:\V\rightarrow \V$ is said to be an automorphism of $K$ if $A(K)=K$.
\end{definition}
Since $K$ has nonempty interior, an automorphism of $K$ is necessarily invertible. We denote the set of all automorphisms of $K$ by $\mathrm{Aut(K)}$ and write $\mathrm{Aut(\V)}$ as an abbreviation for $\mathrm{Aut(\V,\R^n,\lambda)}$.

\gap

We now record some properties in the semi-FTvN system $(\V,\R^n,\lambda)$ induced by $p$ and $e$.
\begin{itemize}
\item $\lambda(e)=\one$ (the vector of $1$s in $\R^n$) and $\lambda(x+re)=\lambda(x)+r\one$ for all $r\in \R$.
\item $\lambda(x)=\one\Rightarrow x=e$ (by the completeness of $p$).
\item $A\in \mathrm{Aut(\V)}\Rightarrow A\in \mathrm{Aut(K)}$ and $Ae=e$. 
\item $\langle x,e\rangle =\langle \lambda(x),\lambda(e)\rangle$\,\, \mbox{for all} \,\,$x\in \V$. So, $e$ strongly commutes with every $x$ in $\V$.
\item ${\cal C}=\R\,e$. As in \cite{gowda-jeong-ftvn}, Example 4.8, this can be seen as follows: From above, $\R\,e\subseteq {\cal C}$. If $x\in {\cal C}$, then $x$ commutes with $-x$; hence, by Proposition \ref{commutativity equivalence},
$\lambda(x)+\lambda(-x)=\lambda(x+(-x))=\lambda(0)=0.$ Thus, $\lambda(x)=-\lambda(-x)$. Since the entries of $\lambda(x)$ are decreasing and those of $-\lambda(-x)$ are increasing, $\lambda(x)$ must be a multiple of $\one.$ By completeness of $p$, $x$ is a multiple of $e$. Hence, ${\cal C}=\R\,e$.
\item $D+D^T=0$ for every $D\in \mathrm{Lie(\Aut(\V))}$.
\item $De=0$ for every $D\in \mathrm{Lie(\Aut(\V))}$.
(We recall that $e$ strongly commutes, hence, commutes with every $x\in \V$.)
\end{itemize}

\gap

 One way of interpreting condition (\ref{ftvn defn}) is: For every $c,u\in \V$, there is an $x$ in the $\lambda$-orbit of $u$ that commutes with $c$. Based on the commutative property (see e.g., Item $(d)$ in Proposition \ref{commutativity equivalence}), it was shown in \cite{gowda-ftvn}, Corollary 2.8, that when $(\V,\W,\lambda)$ is an FTvN system, $\lambda(V)$ is a convex cone in $\W$ (additionally closed when $\V$ is finite-dimensional). The following example shows that such a property may not hold when $(\V,\W,\lambda)$ is merely a semi-FTvN system.

\begin{example} \label{subspace of a FTvN} This example is taken from \cite{bauschke et al}. Consider the FTvN system $(\R ^3,\R ^3,\lambda)$ with $\lx=x^\downarrow$. This corresponds to the complete hyperbolic polynomial $p(x)=x_1x_2x_3$, where $x=(x_1,x_2,x_3)$. Let $\mathcal{U}$ be the span of vectors  $(1, 1, 1)$ and $(3, 1,0)$. 
For any $x=\alpha (1, 1, 1) + \beta (3, 1, 0)$ in $\mathcal{U}$ with $\alpha,\beta\in \R$, we see that $\lambda(x)=x$ if $\beta \geq 0$ and $\lx= \alpha (1, 1, 1) + \beta (0, 1, 3)$ if $\beta<0$.
As $(\R^3,\R^3,\lambda)$ is a FTvN system, $(\U,\R^3,\lambda)$ is a semi-FTvn system. It has been observed in \cite{bauschke et al}, Example 5.2 that the range of $\lambda$ is not convex, hence $p$ (restricted to $\U$) is not isometric. 
Additionally, condition (\ref{ftvn defn}) fails to hold, e.g., with  $c=(3,1,0)$, $u=-(3,1,0)$, and $\lu=(0,-1,-3)$. 
In conclusion, $(\U,\R^3,\lambda)$ is a {\it semi-FTvN system in which  $\lambda(\U)$ is not convex}.
This example also shows that {\it a subspace of an FTvN system need not be a FTvN system.}
\end{example}
\gap

\noindent{\bf Remarks.} In a FTvN system $(\V,\W,\lambda)$, the {\it sublinearity relation}
$$\langle \lambda(c),\lambda(x+y)\rangle \leq \langle\lambda(c),\lambda(x)\rangle +\langle \lambda(c),\lambda(y)\rangle\quad (c,x,y\in \V)$$
holds. Since such a relation holds for hyperbolic polynomials (\cite{bauschke et al}, Corollary 3.5), it continues to hold in a semi-FTvN system that comes from a complete hyperbolic polynomial. Whether such a relation holds in a general semi-FTvN system is unclear.

\gap

We end this section with a characterization of strong commutativity in the semi-FTvN system induced by a complete hyperbolic polynomial.

\begin{theorem}
Consider the semi-FTvN system induced by a complete hyperbolic polynomial. Then, $a$ and $b$ strongly commute in this system if and only if 
$$[\lambda(a)-\lambda(b)]^\downarrow =\lambda(a-b).$$
\end{theorem}

\begin{proof}
Suppose $[\lambda(a)-\lambda(b)]^\downarrow =\lambda(a-b).$ Then, $||\lambda(a)-\lambda(b)||=||\lambda(a-b)||=||a-b||.$ By Proposition \ref{commutativity equivalence}, $a$ and $b$ strongly commute.\\
Now suppose $a$ and $b$ strongly commute so that (by Proposition \ref{commutativity equivalence}), 
$||\lambda(a)-\lambda(b)||=||a-b||.$ By (\ref{lidskii inequality}) and the Hardy-Littlewood-Polya theorem (\cite{bhatia}, Theorem II.1.10), there exists a doubly stochastic matrix $A$ such that 
$$\lambda(a)-\lambda(b)=A\lambda(a-b).$$
Moreover, by Birkhoff's theorem (\cite{bhatia}, Theorem II.2.3), $A$ could be written as a convex combination of permutation matrices. Thus,
$$\lambda(a)-\lambda(b)=\Big (\sum_{k=1}^{N}t_kP_k\Big )\lambda(a-b),$$
where $t_k$s are positive with sum $1$ and $P_k$s are permutation matrices. We claim that $N=1$.
Now, $$||a-b||=||\lambda(a)-\lambda(b)||=||\Big (\sum_{k=1}^{N}t_kP_k\Big )\lambda(a-b)||\leq$$
$$ \sum_{k=1}^{N}t_k||P_k\big (\lambda(a)-\lambda(b)\big )||=
\sum_{k=1}^{N}t_k||\lambda(a)-\lambda(b))||=\sum_{k=1}^{N}t_k||a-b||=||a-b||.$$
By the strict convexity of the (usual) norm, $N=1$. This shows that $\lambda(a)-\lambda(b)=P_1\big (\lambda(a-b)\big )$, hence, $[\lambda(a)-\lambda(b)]^\downarrow =\lambda(a-b).$ This completes the proof.
\end{proof}

\noindent{\bf Remarks.} Consider a Euclidean Jordan algebra which corresponds to a complete hyperbolic polynomial. In this setting, the above result says that $[\lambda(a)-\lambda(b)]^\downarrow =\lambda(a-b)$ if and only if $a$ and $b$ strongly operator commute. When specialized to the algebra of complex Hermitian matrices, this statement yields a result of Massey et al \cite{massey et al.}. 

\section{Commutation principles}

In this section, we describe some commutation principles where the focus is on commutativity instead of strong commutativity. Our goal is to extend the following commutation principle of Ram\'irez et al., from Euclidean Jordan algebras to group/semi-FTvN settings. 

\begin{theorem} (Ram\'irez, Seeger, and Sossa \cite{ramirez et al}) Let $\V$ be a Euclidean Jordan algebra, $E$ be a spectral set in $\V$, and
$F:\V\rightarrow \R$ be a spectral function. Let $\Theta:\V\rightarrow \R$ be Fr\'{e}chet differentiable. If $a$ is a
local minimizer/maximizer of the map
$$ x\in E \rightarrow  \Theta(x)+F(x),$$
then $a$ and $\Theta^{\prime}(a)$ operator commute.
\end{theorem}

In \cite{gowda-jeong-siopt}, Theorem 1.2, Gowda and Jeong extended the above result by assuming that $E$ and $F$  are weakly spectral, see below for the definitions. Their short proof (see also the proof of Theorem 2.1 in \cite{seeger-orbital}) made use of the Lie algebra of the automorphism group of $\V$. Now, we extend this further to group settings.  
\\

In the following, we will formally describe Theorem \ref{theorem in introduction} and provide its proof. We rely on the notation used in the Introduction. We recall that for any map $h:\V\rightarrow \V$ and $E\subseteq \V$, the  {\it variational inequality problem} $\mathrm{VI}(h,E)$ is  to
\begin{center}
find $x^*\in E$ such that $\langle h(x^*),x-x^*\rangle \geq 0$ for all $x\in E$
\end{center}
and for a nonempty set  $S$ in $\V$, the {\it normal cone of $S$} is given by
$$N_{S}(a):=\{d\in \V:\langle d,x-a\rangle\leq 0,\,\,\forall\,\,x\in S\}.$$

\begin{theorem}\label{group commutation principle}
     Let $\V$ be a real inner product space, ${\cal G}$ be a closed subgroup of the group of invertible bounded linear transformations on $\V$. Let  
     \begin{itemize}
         \item $E$ be set in $\V$ that is ${\cal G}$-invariant (so,  $A(E)\subseteq E$ for all $A\in {\cal G}$),
         \item Let
$F:\V\rightarrow \R$ be a real-valued function that is ${\cal G}$-invariant (so, $F(Ax)=F(x)$ for all $A\in {\cal G}$ and $x\in\V$),  
\item $\Theta:\V\rightarrow \R$ be Fr\'{e}chet differentiable, 
\item $h:\V\rightarrow \V$, 
\item $a\in E$ and $c\in \V$.
\end{itemize}
Then the following statements hold:
\begin{itemize}
    \item [$(i)$] If $a$ is a
local minimizer/maximizer of the map
$$ x\in E \rightarrow  \Theta(x)+F(x),$$
then $a$  commutes with $\Theta^{\prime}(a)$ relative to ${\cal G}$.  
    \item  [(ii)] If $a$ is an optimizer of $$\underset{x\in E}{max/min}\quad \langle c,x\rangle,$$ then $a$  commutes with $c$ relative to ${\cal G}$;
\item  [(iii)] If $a$ is a solution of the of the variational inequality problem $\mathrm{VI}(h,E)$,
then $a$  commutes with $h(a)$ relative to ${\cal G}$.
\item [$(iv)$] Relatively to ${\cal G}$, $a$ commutes with every element in the normal cone of $E$. 
\end{itemize}
\end{theorem}

 \begin{proof}
$(i)$ First, consider the case when $a$ is a local
minimizer of $\Theta+F$ over $E$. 
Then we have
$$\Theta(a)+F(a)\leq \Theta(x)+F(x)\,\,\mbox{for all}\,\,x\in B_a\cap E,$$
where $B_a$ denotes some open ball around $a$. Let $D\in \mathrm{Lie({\cal G})}$ so that for all  $t\in \R$, $\exp(tD)\in \mathrm{{\cal G}}$. As $E$ and $F$ are ${\cal G}$-invariant, we have 
$\exp(tD)a\in E$ and $F(\exp(tD)a)=F(a)$ for all $t\in \R$. Moreover, by the continuity of $\exp(tD)$ in $t$, $\exp(tD)a\in B_a\cap E$ for all $t$ close to zero.
For such $t$,
$$\Theta(a)+F(a)\leq \Theta\big (\exp(tD)a\big)+F\big(\exp(tD)a\big)=\Theta\big(\exp(tD)a\big)+F(a)$$
and so $\Theta(a)\leq \Theta(\exp(tD)a)$. Then, the  derivative of $\phi(t):=\Theta(\exp(tD)a)$ at $t=0$ is zero. Since this derivative is $\langle \Theta^{\prime}(a), Da\rangle$, we have $\langle Da,\Theta^\prime(a)\rangle=0$. As $D\in \mathrm{Lie({\cal G})}$ is arbitrary, $a$  commutes with $\Theta^{\prime}(a)$ relative to ${\cal G}$.
\\
Now we consider the (local) maximum case. By considering $-(\Theta+F)$, we see that $a$  commutes with  $-\Theta^\prime(a)$ relative to ${\cal G}$. Hence, $a$  commutes with   $\Theta^\prime(a)$ relative to ${\cal G}$.\\
$(ii)$
We specialize $(i)$ by taking $F=0$ and $\Theta(x):=\langle c,x\rangle$. Then $a$  commutes with $c$ relative to ${\cal G}$. \\
$(iii)$ If  $\langle h(a),x-a\rangle \geq 0$ for all $x\in E$, then $a$ solves the optimization problem $\underset{x\in E}{\min} \,\langle c,x\rangle$, where $c:=h(a)$. By $(ii)$, $a$ commutes with  $h(a)$ relative to ${\cal G}$. This completes the proof.\\
$(iv)$ if $c$ is in the normal cone of $E$ at $a$, then $\langle c,x-a\rangle \leq 0$ for all $x\in E$. By $(ii)$, 
$a$ commutes with $c$ relative to ${\cal G}$.
\end{proof}

\smallgap
\begin{corollary} \label{convex corollary}
 Let the assumptions be as in the above theorem. Suppose further that 
 \begin{itemize}
     \item $\V$ is finite-dimensional, 
     \item $E$ is convex, 
     \item 
     $f:\V\rightarrow \R$ is a convex function and 
 $f(a)=\underset{x\in E}{\min}\, f$.
 \end{itemize}
 Then,  relative to ${\cal G}$, $a$  commutes with some element in the subdifferential of $f$ at $a$.
     \end{corollary}

     \begin{proof} 
     
     Let $g$ denote the indicator function of $E$ (so it takes the
value zero on $E$ and infinity outside of $E$). Then, $a$ is an optimizer of the (global)
convex problem $\min\,(f+g)$ and so
$$0\in \partial(f+g)=\partial f(a)+\partial g(a),$$
where the equality comes from the subdifferential sum formula \cite[Theorem
23.8]{Rock}. Hence, there is a $c\in \partial\,f(a)$ such that $-c\in \partial g(a)$.  This $c$ will have the
property that
$\langle -c,x-a\rangle\leq 0$ for all $x\in E$,
that is, $-c$ belongs to the normal cone of $E$ at $a$. By the above theorem, $a$ commutes with $-c$ (equivalently, $c$) relative to ${\cal G}$.
\end{proof}

\smallgap

\begin{example} Consider a Euclidean Jordan algebra $\V$ with its symmetric cone $K$. Let $\Aut(K)$ denote the automorphism group of $K$, see Example 2.5. In view of Example 2.6, we know how to describe commutativity relative to $\Aut(K)$. Now suppose $f$ is a convex function on $\V$ and $a\in K$ with $f(a)=\underset{K}{\min}\,f$. As $K$ is invariant under the group $\Aut(K)$, the above corollary shows that, relative to this group, $a$ commutes with some element $c$ in subdifferential of $f$ at $a$.  Thanks to  Example 2.6, we infer that $a$ and $c$ operator commute and $a\circ c=0$.  Thus, along with the operator commutativity  - which can be perceived as a condition coming from the automorphism group $\Aut(\V)$ - we now have an additional optimality condition, namely, $a\circ c=0$.   
 In particular, when $f$ is differentiable, $a$ operator commutes with $f^{\prime}(a)$ and $a\circ f^{\prime}(a)=0$.
\end{example}

\smallgap

The following result is motivated by similar results in the setting of Euclidean Jordan algebras \cite{gowda-jeong-siopt, jeong-sossa-arxiv2024}.

\begin{theorem}
    Let $\V$, ${\cal G}$, $E$, and $F$ be as in the above theorem. Let $\Phi:E\rightarrow \R$ be a real-valued function. 
     If $a$ is a
local maximizer of the map
$$ x\in E \mapsto \Phi(x)+F(x),$$ then, relative to ${\cal G}$,
$a$ commutes with every element in the subdifferential $\partial_{E}\,\Phi(a)$, 
where
\begin{equation}\label{subdifferential condition}
\partial_{E}\,\Phi(a):=\{d\in \V: \Phi(x)-\Phi(a)\geq \langle d,x-a\rangle\,\,\forall\,\,x\in E \}
\end{equation} is assumed to be nonempty.
\end{theorem}

\begin{proof}
We fix $d\in \partial_{E}\,\Phi(a)$ and $D\in \mathrm{Lie({\cal G})}$. We need to show that $\langle Da,d\rangle=0$. \\As $a$ is a local maximizer, there is an open ball $B_a$ around $a$ such that 
$$\Phi(a)+F(a)\geq \Phi(x)+F(x)\,\,\mbox{ for all}\,\, x\in B_a\cap E.$$
By continuity, there is a $\delta>0$ such that 
$t\in (-\delta,\delta)\Rightarrow \exp(tD)a\in B_a.$  From the assumed conditions on $E$ and $F$,  $\exp(tD)a\in E$ and $F\big(\exp(tD)a\big)=F(a)$; so we have
$$t\in (-\delta,\delta)\Rightarrow \Phi(a)+F(a)\geq \Phi\big (\exp(tD)a\big)+F\big(\exp(tD)a\big)=\Phi\big (\exp(tD)a\big)+F(a).$$
By canceling $F(a)$ and using (\ref{subdifferential condition}),
$$0\geq \Phi\big (\exp(tD)a\big)-\Phi(a)\geq \langle d, \exp(tD)a-a\rangle\,\,\mbox{for all}\,\, t\in (\delta,\delta).$$
This leads to $\langle d,\exp(tD)a\rangle \leq \langle d,a\rangle$ for all $t\in (-\delta,\delta)$. Considering the   derivative of $\phi(t):=\langle d,\exp(tD)a\rangle $ at $t=0$, we get $\langle Da,d\rangle=0$. 
\end{proof}

\smallgap

For comparison purposes, our next result and the following examples are stated in the setting of semi-FTvN systems. First, we introduce some spectral notions.

 \begin{definition}
 Let $(\V,\W,\lambda)$ be a semi-FTvN system, ${\cal G}$ be a closed subgroup of $\mathrm{Aut(\V)}$ $(=\mathrm{Aut(\V,\W,\lambda)})$, $E$ be a subset of $\V$, and $F:\V\rightarrow \R$. We say that 
     \begin{itemize}
         \item [(i)] 
      $E$ is  a spectral set if it is of the form $\lambda^{-1}(Q)$ for some $Q\subseteq \W$;
      \item [(ii)] $E$  is  weakly spectral if it is invariant under all automorphisms, that is, $A(E)\subseteq E$ for all $A\in \mathrm{Aut(\V)};$
      \item [(iii)] $E$  is  weakly spectral relative to ${\cal G}$ if  $A(E)\subseteq E$ for all $A\in {\cal G};$
      \item [(iv)] $F$ is a spectral function if it is of the form $F=f\circ \lambda$ for some $f:\W\rightarrow \R$;
      \item [(v)] $F$ is  weakly spectral if $F(Ax)=F(x)$ for all $x\in \V$ and $A\in \Aut(\V)$;
      \item [(vi)] $F$ is said to be weakly spectral relative to ${\cal G}$ if $F(Ax)=F(x)$ for all $x\in \V$ and $A\in {\cal G}$.
\end{itemize}
\end{definition}
We remark that every spectral set/function is weakly spectral (relative to ${\cal G})$. \\

In a semi-FTvN system $(\V,\W,\lambda)$, for any $a\in \V$, we define
{\it orbit, weak orbit, and ${\cal G}$-orbit of $a$ by} 
$$[a]:=\{x\in \V: \lambda(x)=\lambda(a)\},\,\,
[a]_w=\{Aa:A\in \mathrm{Aut(\V)}\},\,\,\mbox{and}\,\,[a]_{{\cal G}}=\{Aa:A\in {\cal G}\}.$$
Then $[a]$ is spectral,  $[a]_w$ is weakly spectral, and $[a]_{{\cal G}}$ is weakly spectral relative to ${\cal G}$. 
We always have $[a]_{{\cal G}}\subseteq [a]_w\subseteq [a]$.
\gap

 We now specialize Theorem \ref{group commutation principle} to semi-FTvN systems. 
 
 \begin{corollary} \label{VI}
     Let $(\V,\W,\lambda)$ be a semi-FTvN system, ${\cal G}$ be a closed subgroup of $\mathrm{Aut(\V,\W,\lambda)}$. Suppose, relative to ${\cal G}$, $E$ be a weakly spectral set in $\V$, and
$F:\V\rightarrow \R$ be a weakly spectral function. Let $\Theta:\V\rightarrow \R$ be Fr\'{e}chet differentiable. Then the following hold: If $a$ is a
local minimizer/maximizer of the map
$$ x\in E \rightarrow  \Theta(x)+F(x),$$
then $a$ and $\Theta^{\prime}(a)$ commute relative to ${\cal G}$; In particular,
if $a$ is an optimizer of the problem
$$ \underset{x\in E}{max/min}\, \langle c,x\rangle,$$ then $a$ and $c$ commute relative to ${\cal G}$; Furthermore, if $a$ solves the variational inequality problem $\mathrm{VI}(h,E)$, then $a$ and $h(a)$  commute relative to ${\cal G}$.
\end{corollary}

 \smallgap

 \begin{example} 
 Suppose $(\V,\W,\lambda)$ is a semi-FTvN system, where $\V$ is finite dimensional. Then, for each $u\in \V$, the spectral set $[u]$ is closed and bounded (recall $\lambda$ is continuous and an isometry), hence compact. Then, for any $c\in \V$, $\underset{[u]}{\max}\, \langle c,x\rangle$ is attained, say, at $a\in [u]$. Then, $a$  commutes with $c$. We remark that when $(\V,\W,\lambda)$ is an FTvN system, one gets strong commutativity (not just commutativity), see  Section \ref{intro} and \cite[Section 1]{gowda-ftvn}. Also, in this setting, when $E$ is a spectral set and $h:E\rightarrow \V$,  if $a$ solves $\mathrm{VI}(h,E)$, then $a$ and $-h(a)$ strongly commute. This was established in  \cite[Corollary 1]{gowda-commutation-ORL}.\\
In a similar vein (analogous to Corollary \ref{convex corollary}), we have the following from \cite{gowda-ftvn}, 
Proposition 3.12: Suppose $(\V,\W,\lambda)$ is a 
finite-dimensional FTvN system, $E$ is a convex spectral set in $\V$, $f:\V\rightarrow \R$ is convex and $a$
is an optimizer of the problem $\underset{E}{\min}\,f$. Then, $a$ strongly commutes with $-c$ for some element $c$ in the subdifferential of $f$ at $a$.
\end{example}

\begin{example}
Suppose $(\V,\W,\lambda)$ is a semi-FTvN system, $K$ be a closed convex cone in $\V$ that is weakly spectral. (This could be the hyperbolicity cone corresponding to a complete hyperbolic polynomial.) Let $h:K\rightarrow \V$. Consider the {\it complementarity problem} $\mathrm{CP}(h,K)$, which is to find 
\begin{center}
$x^*\in K$ such that $h(x^*)\in K^*$ and $\langle h(x^*),x^*\rangle=0$,
\end{center}
where $K^*:=\{y\in \V: \langle y,x\rangle\geq 0,\,\forall\,x\in K\}$ is the dual of $K$.
If $x^*$ is such a solution, by Corollary \ref{VI} specialized to $E=K$, we see that $x^*$  commutes with $h(x^*)$. In particular, in the setting of a Euclidean Jordan algebra with its symmetric cone $K$, we get the operator commutativity of $x^*$ and $h(x^*)$.
\end{example}
\gap

\smallgap

\smallgap

\noindent{\bf Concluding remarks.} In this paper, we introduced two commutativity concepts in the settings of transformation groups and semi-FTvN systems, and demonstrated their appearance as optimality conditions in certain optimization problems. Additionally, we have shown that these concepts reduce to the well-known operator commutativity concepts in the setting of Euclidean Jordan algebras.  In our future work, we plan to go beyond hyperbolic polynomials to address majorization concepts/results in the setting of semi-FTvN systems. 

\section*{Acknowledgements}
D. Sossa acknowledges the support of  FONDECYT (Chile)  through grant 11220268, and MATH-AMSUD 23-MATH-09 MORA-DataS project.


\end{document}